\documentclass[hidelinks,onefignum,onetabnum]{siamart251216}

\usepackage{lipsum}
\usepackage{amsfonts}
\usepackage{graphicx}
\usepackage{epstopdf}
\ifpdf
  \DeclareGraphicsExtensions{.eps,.pdf,.png,.jpg}
\else
  \DeclareGraphicsExtensions{.eps}
\fi

\newsiamremark{remark}{Remark}
\newsiamremark{hypothesis}{Hypothesis}
\crefname{hypothesis}{Hypothesis}{Hypotheses}
\newsiamthm{claim}{Claim}
\newsiamremark{fact}{Fact}
\crefname{fact}{Fact}{Facts}

\headers{An Example Article}{D. Doe, P. T. Frank, and J. E. Smith}

\title{An Example Article\thanks{Submitted to the editors DATE.
\funding{This work was funded by the Fog Research Institute under contract no.~FRI-454.}}}

\author{Dianne Doe\thanks{Imagination Corp., Chicago, IL 
  (\email{ddoe@imag.com}, \url{http://www.imag.com/\string~ddoe/}).}
\and Paul T. Frank\thanks{Department of Applied Mathematics, Fictional University, Boise, ID 
  (\email{ptfrank@fictional.edu}, \email{jesmith@fictional.edu}).}
\and Jane E. Smith\footnotemark[3]}

\usepackage{amsopn}

\ifpdf
\hypersetup{
  pdftitle={Asymptotic and Finite Time Analysis of Stochastic Markov Mirror Descent Algorithm with Applications in Reinforcement Learning},
  pdfauthor={Anik}
}
\fi

\title{Stochastic Mirror Descent under Iterate-Dependent Markov Noise: Analysis in the Asymptotic and Finite Time Regimes\thanks{Submitted to the editors DATE.\\The authors are with the Department of Computer Science and Automation, Indian Institute of Science, Bengaluru 560012. Email: anikpaul42@gmail.com, shalabh@iisc.ac.in.}
\funding{This work was funded by the Walmart Centre for Tech Excellence, IISc; by the J. C. Bose Grant with No. ANRF/JBG/2025/000209/HAA from ANRF, Government of India; and by the Prof.\ B. S. Sonde Chair Professorship from IISc.}}

\author{Anik Kumar Paul
\and  Shalabh Bhatnagar }

\usepackage{float}
\usepackage[labelfont=bf]{caption}
\usepackage{makecell}
\usepackage{multirow} 
\usepackage{xspace}
\usepackage{pifont}
\usepackage{xurl} 
\usepackage{enumitem}
\usepackage{threeparttable}
\usepackage{hyperref} 

\usepackage{graphicx}
\usepackage{booktabs}
\usepackage{diagbox}
\usepackage{cite}
\usepackage{graphicx}
\usepackage{graphics}
\usepackage{textcomp}
\usepackage{epsfig}
\usepackage{times}
\usepackage{amsmath}
\usepackage{amssymb}
\usepackage{amsfonts}
\usepackage{float}
\usepackage{siunitx}
\usepackage{scalerel}
\usepackage{cite}
\usepackage{textcomp}
\usepackage{stackengine}
\usepackage{mathtools}
\usepackage{caption}
\usepackage{subcaption}
\usepackage{nopageno}
\usepackage{comment}

\usepackage{booktabs} 
\usepackage{float}
\usepackage{placeins}
\usepackage[skip=3pt]{caption}

\usepackage{graphicx}
\usepackage{bm}
\usepackage{booktabs}

\newcommand{\ncom}{\newcommand}

\newcommand{\beqn}{\begin{eqnarray*}}
\newcommand{\eeqn}{\end{eqnarray*}}
\newcommand{\beq}{\begin{eqnarray}}
\newcommand{\eeq}{\end{eqnarray}}

\newcommand{\norm}[1]{\left\lVert #1 \right\rVert}
\newcommand{\inprod}[2]{\left\langle #1, #2 \right\rangle}
\ncom\R{\mathbb{R}}

\DeclareMathOperator*{\argmin}{arg\,min}

 \newtheorem{assumption}{Assumption}
 
\usepackage{algpseudocode}
\usepackage{tikz}

\begin{document}

\maketitle

\begin{abstract}
We study a stochastic optimization problem in which the sampling distribution depends on the decision variable, and the available samples are generated through an iterate-dependent Markov chain. Such settings arise naturally in problems with decision-dependent uncertainty; however, they introduce bias and temporal dependence, which render standard techniques developed for i.i.d.\ noise inapplicable.
In this work, we analyze the stochastic mirror descent algorithm under iterate-dependent Markov noise. We first establish almost sure convergence for both convex and non-convex problems under the mild assumption of Lipschitz continuity of the objective function, without requiring differentiability. We then derive finite-time concentration bounds for smooth objectives.
In the convex setting, the resulting sample complexity matches the classical rate of stochastic mirror descent under i.i.d.\ noise. In the non-convex setting, we obtain a sample complexity bound in terms of the norm of the Riemannian gradient over the probability simplex. 
Overall, our results establish a unified convergence framework for stochastic mirror descent with state-dependent Markov noise, and highlight its behavior in both convex and non-convex regimes.
\end{abstract}


\section{Introduction}

We consider the stochastic optimization problem
\begin{equation}
    \min_{x \in \mathcal{X}} f(x)
    := \mathbb{E}_{s \sim \mu_x}[F(x,s)]
    = \int_{\mathcal{S}} F(x,s)\, d\mu_x(s),
    \label{eq:main_problem1}
\end{equation}
where $\mathcal{X} \subseteq \mathbb{R}^d$ is a compact and convex constraint set. 
The function $F : \mathbb{R}^d \times \mathcal{S} \to \overline{\mathbb{R}}$ denotes the stochastic objective. 
Further, $(\mathcal{S}, \mathcal{B}(\mathcal{S}))$ is a measurable space, where $\mathcal{S}$ is a topological space equipped with its Borel $\sigma$-algebra. 
For each decision variable $x \in \mathcal{X}$, the probability measure $\mu_x$ characterizes the distribution of the stochastic state $s$.

\medskip

The formulation in \eqref{eq:main_problem1} extends the classical stochastic optimization framework, where the underlying randomness is typically modeled through a fixed distribution that does not depend on the decision variable \cite{Duchi2018IntroductoryLO,bubeck2011}. 
In that standard setting, one considers objectives of the form
\[
\mathbb{E}_{\zeta \sim \mathcal{P}}[F(x,\zeta)],
\]
where the distribution $\mathcal{P}$ remains independent of $x$. 
In contrast, in \eqref{eq:main_problem1}, the distribution $\mu_x$ varies with the decision variable $x$. 
This feature allows the model to capture state-dependent or decision-dependent uncertainty.

Such a formulation naturally appears in several applications, including reinforcement learning \cite{halperin2022reinforcement}, controlled Markov processes \cite{borkar2006stochastic}, and performative prediction \cite{cutler2024stochastic}. 
In these settings, the data-generating distribution evolves in response to the decisions made by the algorithm, making the dependence on $x$ intrinsic to the problem.
A standard assumption in such environments is that direct sampling from $\mu_x$ is not available. 
Instead, for each decision variable $x$, one simulates a Markov chain whose stationary distribution is $\mu_x$. 
This dependence on $x$ introduces significant analytical challenges.

In particular, the associated stochastic oracle is no longer unbiased in the classical sense. 
Moreover, the resulting noise sequence is  typically sampled through a Markov process. 
Consequently, many of the standard tools developed for stochastic optimization under i.i.d.\ noise are not directly applicable. 
This necessitates the development of new techniques to analyze the convergence and stability properties of the resulting algorithms.

Despite its theoretical importance and wide range of applications, the convergence and finite-time analysis of stochastic optimization under Markov noise has received attention only relatively recently. 
For instance, \cite{sun2018markov,even2023stochastic,doan2020convergence} study the asymptotic and finite-time behavior of stochastic gradient descent and its accelerated variants in settings where direct sampling from the target distribution is not available, and instead samples are obtained from an aperiodic and irreducible Markov chain with a unique stationary distribution. 
The analysis has also been extended to non-convex and unconstrained problems in \cite{kar2026high}.

However, these works are largely restricted to state-independent transition kernels, where the stationary distribution remains fixed. 
The state-dependent Markov setting has been studied in \cite{che2026stochastic}, which establishes convergence and convergence rates for stochastic gradient descent with iterate-dependent transition kernels in the unconstrained setting. 
We refer to \cite{roy2022constrained} for the analysis of non-convex and smooth optimization problems in the constrained setting under Markov noise.

\medskip

While stochastic gradient descent is naturally suited for Euclidean spaces, many practical problems involve constraints that induce non-Euclidean geometry. 
This motivates the use of Mirror Descent, which leverages a distance-generating function to induce a Bregman divergence and to better align the algorithm with the geometry of the constraint set \cite{rsa}. 
In the Markov setting, \cite{duchi2012ergodic} analyzes mirror descent under state-independent transition kernels, and \cite{solodkin2026methods} further studies this framework along with accelerated variants under the assumption of a strongly monotone gradient map. 

\medskip

However, to the best of our knowledge, there is no existing work on stochastic mirror descent with iterate-dependent transition kernels, either for convex or non-convex problems. Furthermore, even within the existing literature on standard stochastic gradient methods, most analyses establish convergence only in expectation (e.g., $L_1$ or $L_2$ convergence), ensuring that the average behavior across infinitely many runs converges. This is a significant practical limitation, as the expected convergence does not guaranty convergence during a single run \cite{liu2024almost}. Although almost sure convergence has been analyzed for unconstrained SGD under Markov noise \cite{doucet2017asymptotic}, the non-Euclidean framework required for mirror descent remains entirely unexplored.

\medskip

In this paper, we develop a comprehensive almost sure convergence analysis of stochastic mirror descent with iterate-dependent Markov noise. To the best of our knowledge, this is the first analysis of any stochastic gradient-based method in this Markov setting that handles both non-convex and non-smooth objectives, requiring only Lipschitz continuity rather than differentiability.

\medskip

Furthermore, we establish finite-time concentration bounds for stochastic mirror descent under iterate-dependent Markov noise with smooth objectives. 
For convex functions, the resulting sample complexity matches that of stochastic mirror descent under i.i.d.\ noise \cite{rsa}. 
For non-convex objectives, we derive a sample complexity in terms of the Riemannian gradient norm as defined in \cite{Boumal_2023} for the constrained set probability simplex, which, to the best of our knowledge, is the first such finite-time result for mirror descent, even in the i.i.d.\ noise setting. In Table \ref{tab:comparison}, we provide a comprehensive comparison highlighting the position of our work relative to the existing literature.

\noindent\textbf{Summary of Contributions.}
\begin{itemize}
    \item We establish \emph{almost sure convergence} of stochastic Markov mirror descent for both convex and non-convex problems under only Lipschitz continuity, without requiring differentiability.
    \item We derive \emph{finite-time concentration bounds} for stochastic mirror descent with iterate-dependent Markov noise in the smooth setting. For convex objectives, we show that the \emph{sample complexity matches} that of stochastic mirror descent under i.i.d.\ noise. 

For non-convex objectives, we obtain a \emph{finite-time guarantee in terms of the norm of the Riemannian gradient} over the probability simplex. To the best of our knowledge, this provides the \emph{first finite-time analysis of the mirror descent algorithm} in the presence of Markov noise.
\end{itemize}
\begin{table}[t]
\centering
\small
\setlength{\tabcolsep}{4pt} 
\begin{tabular}{
>{\raggedright\arraybackslash}p{1.3cm}
>{\raggedright\arraybackslash}p{2.0cm}
>{\raggedright\arraybackslash}p{2.4cm}
>{\raggedright\arraybackslash}p{2.8cm}
>{\raggedright\arraybackslash}p{3.2cm}
}
\toprule
\textbf{Alg} & \textbf{Literature} & \textbf{Noise Structure} & \textbf{Regularity} & \textbf{Main Results} \\
\midrule

SGD & \cite{rsa} & i.i.d.\ noise & Convex (smooth/non-smooth) & $L^2$ convergence rate; finite-time concentration bound \\

SGD & \cite{sun2018markov} & Markov (iterate-independent) & Smooth (convex/non-convex) & $L^2$ convergence rate \\

SGD & \cite{doucet2017asymptotic} & Markov (iterate-dependent) & Non-convex (smooth) & Almost sure convergence \\

SGD & \cite{che2026stochastic} & Markov (iterate-dependent) & Convex (smooth); non-convex (smooth) & $L^2$ convergence rate \\

SMD & \cite{rsa} & i.i.d.\ noise & Convex (smooth/non-smooth) & $L^2$ convergence rate; finite-time concentration bound \\

SMD & \cite{duchi2012ergodic} & Markov (iterate-independent) & Convex & $L^1$ convergence of function value \\

SMD & Ours & Markov (iterate-dependent) &  Only Lipschitz (both convex and non-convex) & Almost sure convergence; finite-time concentration bound for smooth function \\

\bottomrule
\end{tabular}
\caption{Comparison of stochastic optimization methods under different noise structures and objective regularity.}
\label{tab:comparison}
\end{table}
\subsection*{Notation}
We denote the real and extended real numbers by $\mathbb{R}$ and $\overline{\mathbb{R}}$, respectively. Let $\mathcal{X} \subseteq \mathbb{R}^d$ be a compact, convex constraint set, and let $(\mathcal{S}, \mathcal{B}(\mathcal{S}))$ denote a measurable space where $\mathcal{S}$ is a topological space. We equip $\mathbb{R}^d$ with a general norm $\|\cdot\|$, and denote its corresponding dual norm by $\|\cdot\|_*$. For a closed convex set $\mathcal{X}$ and a point $x \in \mathcal{X}$, $\mathcal{N}_{\mathcal{X}}(x)$ and $\mathcal{T}_{\mathcal{X}}(x)$ represent the normal and tangent cones at $x$, respectively. The Clarke subdifferential of a locally Lipschitz function $f$ at $x$ is denoted by $\partial f(x)$. 
\section{Problem Setup and Algorithmic Steps}
We consider throughout this paper the following stochastic optimization problem:
\begin{equation}
    \min_{x \in \mathcal{X}} f(x)
    := \mathbb{E}_{s \sim \mu_x}[F(x,s)]
    = \int_{\mathcal{S}} F(x,s)\, d\mu_x(s).
    \label{eq:main_problem}
\end{equation}

We now describe the key features of the problem.

\medskip

\noindent\textbf{Convex and nonconvex objectives.}
We allow the objective function $f$ to be either convex or nonconvex. Importantly, in establishing almost sure convergence of our algorithm, we do not require $f$ to be differentiable. This enables us to handle a broad class of nonsmooth problems.

However, for the finite-time analysis, additional structure is needed--- we assume that the function is continuously differentiable and has Lipschitz gradient. 

\medskip

\noindent\textbf{State-dependent Markov stochastic oracle.}
We consider a setting in which the randomness evolves according to a state-dependent Markov process.

More precisely, suppose that the oracle is currently at a state $s \in \mathcal{S}$. When queried at a point $x \in \mathcal{X}$, the oracle transitions to a new state $s'$ according to a transition kernel $\Pi(x,s)$, and returns the value $F(x,s')$. Formally, for any measurable set $A \subseteq \mathcal{S}$,
\[
    \mathbb{P}(s' \in A \mid x, s) = \Pi(x,s)(A).
\]

For each fixed $x$, we assume that the Markov process induced by $\Pi(x,\cdot)$ is ergodic and admits a unique stationary distribution $\mu_x$. That is, for every measurable set $A \subseteq \mathcal{S}$,
\[
    \mu_x(A)
    = \int_{\mathcal{S}} \Pi(x,s)(A)\, d\mu_x(s).
\]

Under this assumption, the objective function can be interpreted as the long-run average of the stochastic evaluations generated by the oracle. In other words, $f(x)$ represents the steady-state performance associated with the decision $x$.
Next we state the Stochastic  Markov Mirror Descent Algorithm to solve the optimization problem given in \eqref{eq:main_problem}. 

\subsection{Stochastic Markov Mirror Descent Algorithm}

\subsubsection{Bregman Divergence}

Let $R:\mathbb{R}^n \to \overline{\mathbb{R}}$ be a strongly convex function defined on an open set containing $\mathcal{X}$. In the mirror descent framework, this function acts as a mirror map, and it induces the Bregman divergence
\begin{equation}
    \mathbb{D}_R(x,y)
    := R(x) - R(y) - \langle \nabla R(y),\, x-y \rangle.
    \label{Bregman}
\end{equation}

This quantity provides a notion of distance that is tailored to the geometry defined by $R$. Unlike the standard Euclidean distance, the Bregman divergence is, in general, not symmetric in its arguments. Nevertheless, it remains nonnegative, a property that follows directly from the strong convexity of $R$.

We now impose the following structural assumption on the mirror map.

\begin{assumption}
The function $R : \mathbb{R}^d \to \mathbb{R}$ is twice continuously differentiable and $\sigma_R$-strongly convex with respect to the norm $\|\cdot\|$. That is, for all $x,y \in \mathbb{R}^d$,
\begin{equation*}
R(y) 
\ge 
R(x) 
+ \langle \nabla R(x), y - x \rangle 
+ \frac{\sigma_R}{2} \| y - x \|^2.
\end{equation*}
\label{Strong Convexity}
\end{assumption}
\begin{remark}
    For certain standard choices of $\mathcal{X}$ and $R$, such as the probability simplex and the negative entropy mirror map, $R$ may not be differentiable on the boundary of $\mathcal{X}$. In such cases, standard properties of the mirror descent update ensure that if the initialization lies in the relative interior of $\mathcal{X}$, all subsequent iterates strictly remain in the relative interior \cite{bubeck2011}. See Section $5$ for a detailed analysis.  Consequently, the required smoothness and strong convexity hold along the algorithmic trajectory.
\end{remark}
\subsubsection{Stochastic Markov Mirror Descent Algorithm}

We assume access to a stochastic oracle which, for every iterate $x \in \mathcal{X}$ and current Markov state $s$, returns a stochastic Markov subgradient $g(x,s) \in G(x,s).$
The assumptions imposed on the set-valued stochastic Markov subdifferential mapping $G(x,s)$ are stated in the following Assumption.

\begin{assumption}
The set-valued mapping $G:\mathcal{X}\times\mathcal{S}\rightrightarrows \mathbb{R}^d$
is such that, for every fixed $x\in\mathcal{X}$, the mapping $G(x,s)$
is Borel measurable. Furthermore, the Aumann integral of $G(x,\cdot)$ with respect to the invariant measure $\mu_x$ is given by
\begin{equation*}
\int_{\mathcal{S}} G(x,s')\,\mu_x(ds')
=
\left\{
\int_{\mathcal{S}} F(x,s')\,\mu_x(ds')
\;\middle|\;
F(x,\cdot)\in\mathcal{M}(G(x,\cdot)),
\;
F \text{ is } \mu_x\text{-integrable}
\right\}
\end{equation*}
where $\mathcal{M}(G(x,\cdot))$ denotes the collection of measurable selections of the set-valued map $G(x,\cdot)$.
In addition, we assume that
\begin{equation*}
\int_{\mathcal{S}} G(x,s')\,\mu_x(ds')
\subseteq
\partial f(x),
\end{equation*}
where $\partial f(x)$ denotes the Clarke subdifferential of the function $f$ \cite{clarke1990optimization}.
For a detailed discussion on the Aumann integral and measurable selections of set-valued maps, see  \cite{yaji2020stochastic}.
\label{assumpton3}
\end{assumption}
Assumption~\ref{assumpton3} can be viewed as the natural set-valued generalization of the first-order condition commonly employed in stochastic gradient methods with Markov noise for smooth objectives (see Assumption~3.2 in \cite{doucet2017asymptotic} or Equation~(3) in \cite{che2026stochastic}). Furthermore, this structural condition naturally arises in applied settings such as policy gradient methods for average-reward Markov decision processes \cite{li2025stochastic}.
We are now ready to present the Stochastic Markov Mirror Descent Algorithm.
 Let $x_n \in \mathcal{X}$ be the current iterate and let $S_n$ be the current state. The next iterate is obtained as
\begin{equation}
    x_{n+1}
    = \arg\min_{x \in \mathcal{X}}
    \left\{
        \langle g(x_n,S_n), x \rangle
        + \frac{1}{\alpha_n}\mathbb{D}_R(x,x_n)
    \right\},
    \label{mirror_descent1}
\end{equation}
where $\alpha_n>0$ denotes the step-size at instant $n$. 
The state evolves according to
\[
S_{n+1} \sim \Pi(x_n,S_n)(\cdot).
\]
The linear term drives descent in the direction of the stochastic gradient, while the Bregman divergence penalizes large deviations from $x_n$ in the geometry induced by $R$. The strong convexity of $R$ ensures that the above problem admits a unique solution. We need the following further Assumptions which are quite standard in the literature.
\begin{assumption}
The  step-size sequence $\{\alpha_n\}$ is non-increasing and  satisfies
$
\sum\limits_{n \geq 1} \alpha_n = \infty$, 
$\sum\limits_{n \geq 1} \alpha_n^2 < \infty$.
\label{assumptionstepsize}
\end{assumption}
\begin{assumption}
For each $x_n \in \mathcal{X}$, the transition kernel $\Pi(x_n,\cdot)$ is 
ergodic and admits a unique stationary distribution $\mu_{x_n}$.
\end{assumption}
\begin{assumption}
For each $(x,s) \in \mathcal{X} \times \mathcal{S}$, the set $G(x,s)$ is non-empty, convex, and compact.
\label{Rhsn}
\end{assumption}
\begin{assumption}
There exists $\mathbf{G} > 0$ such that
\begin{equation}
    \sup \left\{ \norm{g(x,s)}_\ast \;\middle|\; s \in \mathcal{S}, \; g(x,\cdot) \in \mathcal{M}(G(x,\cdot)) \right\}
    \leq \mathbf{G}.
\end{equation}
\label{Assumption 3}
\end{assumption}
\begin{assumption}
The map
$
G : \mathcal{X} \times \mathcal{S} \rightrightarrows \mathbb{R}^d
$ 
has a closed graph.
\label{Assumption 5}
\end{assumption}
Assumptions~\ref{Rhsn}, \ref{Assumption 3}, and \ref{Assumption 5} ensure that the limiting differential inclusion (cf.\ Theorem~\ref{Stochastic Projected Proposition}) is well-defined and that the corresponding set-valued map satisfies the Marchaud property.
%
The main objective of the next section is to establish that the iterates $\{x_n\}$ converge almost surely to the set of stationary points defined by
\[
S
=
\left\{
x \in \mathcal{X}
\;\middle|\;
\exists \; \zeta \in \partial f(x)
\text{ such that }
\inprod{\zeta}{y-x} \geq 0,
\quad
\forall \; y \in \mathcal{X}
\right\}.
\]
Furthermore, it follows from Lemma~6 of \cite{paul2025zeroth} that if $x^\ast$ is a local minimum of \eqref{eq:main_problem}, then $x^\ast \in S$.
\section{Almost Sure Convergence Analysis}

In this section, we show that the iterates generated by \eqref{mirror_descent1} can be represented as a stochastic approximation scheme with Markov noise in the sense of \cite{yaji2020stochastic}. 
\subsection{Stochastic Approximation with Markov Noise}
 We begin by deriving an \\ equivalent representation of the mirror descent iterates. This form is useful since it reveals the underlying dynamical structure and makes the role of geometry and constraints explicit.

\begin{proposition}
The mirror descent iterates $\{x_n\}$ generated by \eqref{mirror_descent1} admit the following equivalent representation:
\begin{equation}
x_{n+1} - x_n \in -\alpha_n H_2\big(x_n, S_n\big).
\label{stochastic_approximation}
\end{equation}
Here, the set-valued map $H_2 : \mathcal{X} \times \mathcal{S} \rightrightarrows \mathbb{R}^d$ is defined as
\begin{equation}
H_2\big(x_n, S_n\big)
=
\nabla^2 R(x_n)^{-1}
\Big(
G(x_n, S_n) + \widehat{\mathcal{N}}_{\mathcal{X}}(x_n) + b_n
\Big),
\end{equation}
where $b_n=o(\alpha_n)$, and
\begin{equation}
\widehat{\mathcal{N}}_{\mathcal{X}}(x_n)
=
\left\{
\eta \in \mathcal{N}_{\mathcal{X}}(x_n)
\,\middle|\,
\|\eta\|_\ast \le \left(1+\tfrac{L}{\sigma_R}\right) \mathbf{G}
\right\}.
\end{equation}
\label{Stochastic Projected Proposition}
\end{proposition}
\begin{proof}
Fix a step-size $\alpha>0$ and consider the mirror descent update in \eqref{mirror_descent1}. The first-order optimality condition gives
\begin{equation}
\bigg\langle
y_n + \frac{1}{\alpha}\big(\nabla R(x_{n+1}) - \nabla R(x_n)\big),
\, x - x_{n+1}
\bigg\rangle
\geq 0,
\qquad \forall\, x \in \mathcal{X}.
\label{MirrorDescentIneq}
\end{equation}

For simplicity, we take $y_n = g(x_n,S_n)$. By the definition of the normal cone, this implies
\[
-
\Big(
y_n + \frac{1}{\alpha}\big(\nabla R(x_{n+1}) - \nabla R(x_n)\big)
\Big)
\in
\mathcal{N}_{\mathcal{X}}(x_{n+1}).
\]
Equivalently, there exists $\eta_{n+1}(\alpha)\in \mathcal{N}_{\mathcal{X}}(x_{n+1})$ such that
\begin{equation}
-
\Big(
y_n + \frac{1}{\alpha}\big(\nabla R(x_{n+1}) - \nabla R(x_n)\big)
\Big)
=
\eta_{n+1}(\alpha).
\label{eta-def}
\end{equation}

\medskip
\noindent
\textbf{Claim 1.}
There exists a constant $C>0$, independent of $\alpha$, such that
$
\|\eta_{n+1}(\alpha)\| \leq C.
$

\noindent
\emph{Proof of Claim 1.}
Using the $\sigma_R$-strong convexity of $R$, we have
\[
\langle \nabla R(x_{n+1}) - \nabla R(x_n),\, x_{n+1} - x_n \rangle
\geq
\sigma_R \|x_{n+1} - x_n\|^2.
\]
Choosing $x=x_n$ in \eqref{MirrorDescentIneq} yields
\[
\langle
\nabla R(x_{n+1}) - \nabla R(x_n),\, x_{n+1} - x_n
\big\rangle
\leq
\alpha \langle y_n, x_n - x_{n+1} \rangle.
\]
Combining the two inequalities gives
\[
\sigma_R \|x_{n+1} - x_n\|^2
\leq
\alpha \langle y_n, x_n - x_{n+1} \rangle.
\]
Applying now the Cauchy--Schwarz inequality gives,
\[
\sigma_R \|x_{n+1} - x_n\|^2
\leq
\alpha \|y_n\|_\ast\,\|x_{n+1} - x_n\|.
\]
Hence,
\begin{equation}
\|x_{n+1} - x_n\|
\leq
\frac{\alpha}{\sigma_R}\,\|y_n\|_\ast.
\label{eqna}
\end{equation}

Now taking norms in \eqref{eta-def},
\[
\|\eta_{n+1}(\alpha)\|_\ast
\leq
\|y_n\|_\ast
+
\frac{1}{\alpha}
\|\nabla R(x_{n+1}) - \nabla R(x_n)\|_\ast.
\]
Using $L$-smoothness of $R$,
$
\|\nabla R(x_{n+1}) - \nabla R(x_n)\|_\ast
\leq
L\|x_{n+1} - x_n\|.
$ 
Substituting \eqref{eqna}, we obtain 
${\displaystyle
\|\eta_{n+1}(\alpha)\|_\ast
\leq
\Big(1+\frac{L}{\sigma_R}\Big)\|y_n\|_\ast.
}$ 
The claim follows. 
\hfill $\square$

\medskip
\noindent
\textbf{Claim 2.}
The iterate admits the expansion
\[
x_{n+1}
=
x_n
-
\alpha \nabla^2 R(x_n)^{-1}\big(y_n+\eta_{n+1}(\alpha)\big)
+
o(\alpha).
\]

\medskip
\noindent
\emph{Proof of Claim 2.}
From \eqref{eta-def},
\begin{equation}
\nabla R(x_{n+1})
=
\nabla R(x_n)-\alpha y_n-\alpha \eta_{n+1}(\alpha).
\label{eq:Taylor}
\end{equation}
Since $R:\mathbb{R}^d\to\mathbb{R}$ is $C^2$ and strongly convex, the gradient mapping $\nabla R : \mathbb{R}^d \to \mathbb{R}^d$ is injective in $\mathbb{R}^d$, and hence it admits an inverse in its range $\nabla R(\mathbb{R}^d)$ \cite{rockafellar-1970a}. 
Applying $(\nabla R)^{-1}$ to \eqref{eq:Taylor} and performing a first-order Taylor expansion of $(\nabla R)^{-1}$, together with the inverse function theorem, yields
\begin{equation}
x_{n+1}
=
x_n
-
\alpha \nabla^2 R(x_n)^{-1}\big(y_n+\eta_{n+1}(\alpha)\big)
+
o(\alpha).
\label{rumso}
\end{equation}
\hfill $\square$
\noindent
\\ \textbf{Claim 3.}
There exists a sequence $\{\alpha_k\}$ with $\alpha_k\downarrow 0$ such that
\[
\lim_{k\to\infty}\eta_{n+1}(\alpha_k)=\eta
\quad \text{for some } \eta \in \mathcal{N}_{\mathcal{X}}(x_n).
\]

\medskip
\noindent
\emph{Proof of Claim 3.}
Define
\begin{equation}
v(\alpha):=\frac{x_{n+1}-x_n}{\alpha} \leq \frac{1}{\sigma_R}\|y_n\|_\ast, \quad \text{from \eqref{eqna} }
\label{v-alpha}
\end{equation}
 which shows that $\{v(\alpha)\}_{\alpha>0}$ is uniformly bounded. Consequently, for any  sample path, there exists a sequence $\{\alpha_k\}$ with $\alpha_k\downarrow 0$ such that 
$
\lim_{k\to\infty} v(\alpha_k)=v \in \mathbb{R}^d.
$

Next, the update \eqref{mirror_descent1} ensures that $x_{n+1}= x_n$ in the limit $\alpha\downarrow 0$, with $x_{n+1}\in\mathcal{X}$. Therefore,
$
v \in \mathcal{T}_{\mathcal{X}}(x_n).
$ 
By Claim~1, the sequence $\{\eta_{n+1}(\alpha_k)\}$ is bounded. Hence, possibly along a further subsequence $\{k_m\}\subset \{k\}$, there exists $\eta$ such that
$
\lim_{k_m\to\infty}\eta_{n+1}(\alpha_{k_m})=\eta.
$ 
Since $\eta_{n+1}(\alpha_{k_m})\in \mathcal{N}_{\mathcal{X}}(x_{n+1}(\alpha_{k_m}))$ and $x_{n+1}(\alpha_{k_m})= x_n$ in the limit $\alpha_{k_m}\downarrow 0$, the closedness of the normal cone mapping implies $\eta \in \mathcal{N}_{\mathcal{X}}(x_n).$ 
\hfill $\square$

\medskip
\noindent
\textbf{Claim 4.}
The limiting vectors $v$ and $\eta$ satisfy
$
\langle v,\eta\rangle = 0.
$

\noindent
\emph{Proof of Claim 4.}
From the construction,
$
\eta_{n+1}(\alpha_k)\in \mathcal{N}_{\mathcal{X}}\big(x_{n+1}(\alpha_k)\big).
$ 
Thus, for all $y\in\mathcal{X}$,
\[
\langle \alpha_k \eta_{n+1}(\alpha_k),\, y - x_{n+1}(\alpha_k) \rangle \leq 0.
\]
Using
$
x_{n+1}(\alpha_k)=x_n+\alpha_k v(\alpha_k),
$ 
we obtain
$
\langle \alpha_k \eta_{n+1}(\alpha_k),\, y - x_n - \alpha_k v(\alpha_k) \rangle \leq 0.
$ 
Dividing by $\alpha_k$ and choosing $y=x_n$ gives
$
\langle \eta_{n+1}(\alpha_k),\, v(\alpha_k) \rangle \geq 0.
$ 
Passing to the limit as $\alpha_k\downarrow 0$ yields $\langle \eta, v \rangle \geq 0.$
On the other hand, since $\eta \in \mathcal{N}_{\mathcal{X}}(x_n)$ and $v \in \mathcal{T}_{\mathcal{X}}(x_n)$, $\langle \eta, v \rangle \leq 0.$
Combining the two inequalities gives $\langle v,\eta\rangle = 0.$ 
\hfill $\square$
\medskip

\noindent
\textbf{Claim 5.}
The vector $\eta$ satisfies
\begin{equation}
\eta
=
\arg\min_{\zeta \in \mathcal{N}_{\mathcal{X}}(x_n)}
(\zeta + y_n)^\top \nabla^2 R(x_n)^{-1} (\zeta + y_n).
\label{Claim5}
\end{equation}

\medskip
\noindent
\emph{Proof of Claim 5.}
Since $\nabla^2 R(x_n)^{-1}$ is positive definite, \eqref{Claim5} is equivalent to
\begin{equation}
\langle \nabla^2 R(x_n)^{-1}(\eta+y_n),\, \zeta-\eta \rangle \geq 0
\quad \forall\, \zeta \in \mathcal{N}_{\mathcal{X}}(x_n).
\label{eq:claim5}
\end{equation}

    Note that from Claim~2,
$
v = -\nabla^2 R(x_n)^{-1}(y_n+\eta),
$
and from Claim~4, $\langle v,\eta\rangle=0$. Substituting this to the LHS of \eqref{eq:claim5} yields 
\begin{equation*}
    \langle \nabla^2 R(x_n)^{-1}(\eta+y_n),\, \zeta-\eta \rangle =  \langle - \nu,\, \zeta  \rangle \geq 0,
\end{equation*}
which proves the claim.
\hfill $\square$

\medskip
\noindent
\textbf{Claim 6.}
The limit
$
\lim_{\alpha\downarrow 0}\eta_{n+1}(\alpha)=\eta
$ 
exists, where $\eta$ satisfies \eqref{Claim5}. 
Moreover,
\begin{equation}
\|\eta\|_\ast \leq \Big(1+\frac{L}{\sigma_R}\Big)\|y_n\|_\ast.
\label{bound}
\end{equation}

\medskip
\noindent
\emph{Proof of Claim 6.}
From Claim~1, the family $\{\eta_{n+1}(\alpha)\}$ is bounded. Hence, every sequence $\{\eta_{n+1}(\alpha_k)\}$ with $\alpha_k\downarrow 0$ admits a convergent subsequence. 
By Claim~5, any such limit must satisfy \eqref{Claim5}, which uniquely characterizes $\eta$. Therefore, all convergent subsequences have the same limit, and hence
$
\lim_{\alpha\downarrow 0}\eta_{n+1}(\alpha)=\eta.
$ 
The bound \eqref{bound} follows directly from Claim~1. The result then follows from \eqref{rumso}.
\end{proof}
Proposition~\ref{Stochastic Projected Proposition} provides a useful structural interpretation.  One important point to note is that although Proposition~\ref{Stochastic Projected Proposition} shares a similar conceptual foundation with \cite{paul2025convergence}, we develop a detailed self-contained proof based on a different analytical approach. In particular, this approach allows us to represent stochastic mirror descent from the viewpoint of stochastic approximation, which subsequently enables the use of dynamical systems techniques to establish almost sure convergence for any general Lipschitz functions.

 The next corollary will play an important role in the subsequent analysis.
\begin{corollary}\label{TangentConeRepresentation}
Under the same assumptions as in Proposition~\ref{Stochastic Projected Proposition}, the iterates $\{x_n\}$ generated by \eqref{mirror_descent1} admit the equivalent representation
\[
x_{n+1}-x_n=\alpha_n(\nu_n+b_n),
\]
where
\begin{equation}
      \nu_n  =  \argmin_{\nu \in \mathcal{T}_{\mathcal X}(x_n)}
    \left\{
    (\nu+\nabla^2 R(x_n)^{-1}y_n)^\top
    \nabla^2 R(x_n)
    (\nu+\nabla^2 R(x_n)^{-1}y_n)
    \right\}.
    \label{Tangent Cone Representation}
\end{equation}
Here, $y_n=g(x_n,S_n)\in G(x_n,S_n)$ and $b_n=o(\alpha_n)$.
\end{corollary}

\begin{proof}
By Proposition~\ref{Stochastic Projected Proposition}, there exist
$y_n=g(x_n,S_n)$ and $\eta_n\in\mathcal{N}_{\mathcal{X}}(x_n)$ such that
\[
x_{n+1}=x_n-\alpha_n\big(y_n+\eta_n-b_n\big).
\]
Define
\begin{equation}
\nu_n=-\nabla^2 R(x_n)^{-1}y_n-\nabla^2 R(x_n)^{-1}\eta_n.
\label{nundefinition}
\end{equation}

Therefore, $\nu_n$ solves \eqref{Tangent Cone Representation} if and only if the first-order optimality condition holds, that is,
\[
\left\langle
\nabla^2 R(x_n)
\big(
\nu_n +\nabla^2 R(x_n)^{-1}y_n
\big),
\,
\nu_1-\nu_n
\right\rangle
\geq 0,
\qquad
\forall \; \nu_1 \in \mathcal{T}_{\mathcal X}(x_n).
\]

Substituting \eqref{nundefinition} into the left-hand side of the above inequality yields
\[
\left\langle
-\eta_n,
\,
\nu_1-\nu_n
\right\rangle,
\qquad
\forall \; \nu_1 \in \mathcal{T}_{\mathcal X}(x_n).
\]

Note that $\inprod{\eta_n}{\nu_n}=0$ (the proof is similar to Claim-$4$ of Proposition~\ref{Stochastic Projected Proposition}). Therefore, by the definition of the normal cone and the tangent cone, we obtain
\[
\left\langle
-\eta_n,
\,
\nu_1-\nu_n
\right\rangle
\geq 0,
\qquad
\forall \; \nu_1 \in \mathcal{T}_{\mathcal X}(x_n).
\]
This concludes the proof.
\end{proof}
We are now ready to state the main theorem of this section.
\begin{theorem}
The sequence of iterates $\{x_n\}$ generated by the stochastic mirror descent algorithm in \eqref{mirror_descent1} converges almost surely to the set of stationary points $S$.
\end{theorem}
\begin{proof}
The proof proceeds through several steps.

\medskip
\noindent
\textbf{Step 1: Continuous-time interpolation and limiting dynamics.}

We begin by constructing a continuous-time interpolation of the discrete iterates $\{x_n\}$.  
Define the time sequence $\{t_n\}$ by
$
t_0 = 0$,
$t_n = \sum_{k=0}^{n-1} \alpha_k$ and set $I_n = [t_n, t_{n+1}], n\geq 0$. 
We then obtain a continuous time process $\bar{x}(t),t\geq 0$ from $\{x_n\}$ as
\[
\bar{x}(t) 
= 
x_n 
+ 
\frac{t - t_n}{t_{n+1} - t_n}\,(x_{n+1} - x_n),
\qquad \forall\, t \in I_n, n\geq 0.
\]

From Proposition~\ref{Stochastic Projected Proposition}, the mirror descent iterates in \eqref{mirror_descent1} admit the equivalent stochastic approximation form with Markov noise:
\[
x_{n+1} - x_n 
\in 
- \alpha_n \nabla^2 R(x_n)^{-1} 
\Big(
G(x_n, S_n) 
+ \widehat{\mathcal{N}}_{\mathcal{X}}(x_n) 
+ o(\alpha_n)
\Big).
\]

Now, in view of Assumptions~\ref{Rhsn}, \ref{Assumption 3}, and \ref{Assumption 5}, the set-valued map
$
G : \mathcal{X} \times \mathcal{S} \rightrightarrows \mathbb{R}^d
$ 
is Marchaud. Consequently, by \cite{yaji2020stochastic}, the asymptotic behavior of the interpolated trajectory $\bar{x}(t),t\geq 0$ can be characterized through the stability analysis of the following differential inclusion:
\begin{equation}
\begin{aligned}
\dot{x}(t) & 
\in
- \int_{\mathcal{S}} 
\nabla^2 R(x)^{-1} 
\Big(
G(x,s') + \widehat{\mathcal{N}}_{\mathcal{X}}(x)
\Big)
\, \mu_x(ds')
\\ & \subseteq - \nabla^2 R(x)^{-1}  \left( \int\limits_{\mathcal{S}} G(x,s') \mu_x (ds') + \widehat{\mathcal{N}}_\mathcal{X} (x) \right).
\end{aligned}
\label{DI1}
\end{equation}
More precisely, if $L(\bar{x})$ denotes the limit set of the trajectory $\bar{x}(t)$, $t\geq 0$, then
$
L(\bar{x}) \subseteq A,
$ 
where $A$ is the global attractor of the differential inclusion \eqref{DI1}. 
We emphasize that the integral in \eqref{DI1} is understood in the sense of Aumann integration, as defined in Assumption~\ref{assumpton3}.  
Furthermore, by Assumption~\ref{assumpton3}, any Carath\'eodory solution of \eqref{DI1} is also a solution of the reduced differential inclusion
\begin{equation}
\dot{x}(t)
\in -
\nabla^2 R(x)^{-1} 
\Big(
\partial f(x) + \widehat{\mathcal{N}}_{\mathcal{X}}(x)
\Big)
\,.
\label{ODIST}
\end{equation}

Let $A_1$ denote the global attractor of the differential inclusion \eqref{ODIST}. Then it follows immediately that
$
A \subseteq A_1.
$ 
Therefore, to characterize the limit set of $\bar{x}(t)$, it suffices to show that
$
A_1 \subseteq S.
$ 
This will imply that
$
L(\bar{x}(t)) \subseteq S.
$ 

\medskip
\noindent
\textbf{Step 2:}
In this step, we conclude that any Carath\'eodory solution of \eqref{ODIST} is also a solution of the following projected differential inclusion:
{\small 
\begin{equation}
\dot{x}(t)
\in
\mathcal{P}^{x(t)}_{\mathcal{T}_{\mathcal{X}}(x(t))}
\Big(
- \nabla^2 R(x(t))^{-1} \partial f(x(t))
\Big)
\label{PDSdeddi}
\end{equation}
}
and vice versa.
The equation \eqref{PDSdeddi} represents the projected differential inclusion in non-Euclidean domain and the definition of the RHS of \eqref{PDSdeddi} is given below.
{\small 
\begin{equation}
\begin{aligned}
& \dot{x}(t)
 \\ \in  &
\left\{
\argmin_{\nu \in \mathcal{T}_{\mathcal{X}}(x(t))}
\big(\nu + \nabla^2 R(x(t))^{-1} g\big)^\top
\nabla^2 R(x(t))
\big(\nu + \nabla^2 R(x(t))^{-1} g\big)
\;\middle|\;
g \in \partial f(x(t))
\right\}.
\label{PDSdeddi1}
\end{aligned}
\end{equation}
}
This claim can be established using arguments similar to those employed in Corollary~\ref{TangentConeRepresentation}.A related proof in a  setting with any constraint sets but with single-valued maps can be found in \cite{hauswirth2021projected}. We omit the details for brevity.

\medskip
\noindent 
\textbf{Step 3: Final conclusion.}
From Step 2, we conclude that the stability analysis of the differential inclusion in \eqref{ODIST} can be equivalently carried out using the projected differential inclusion in \eqref{PDSdeddi}. 

In Theorem~\ref{stability them}, we show that any Carath\'eodory solution of \eqref{PDSdeddi} asymptotically converges to the set of stationary points $\mathrm{S}$. Consequently, by applying Theorem~6.7 from \cite{yaji2020stochastic}, it follows that 
$
L(\bar{x}) \subseteq \mathrm{S}.
$
\end{proof}

\begin{remark}
The projected differential inclusion in the non-Euclidean setting given in \eqref{PDSdeddi} can be viewed as a natural generalization of the projected dynamical systems in the Euclidean setting \cite{nagurney2012projected}. In particular, the geometry is induced by the Riemannian metric associated with $\nabla^2 R(x)$, and dynamics are driven by a set-valued map. 
For a detailed analysis of projected dynamical systems in non-Euclidean domains, we refer to \cite{hauswirth2021projected}.
\end{remark}

In the next theorem, we show that the set of stationary points coincides with the set of equilibrium points of the projected differential inclusion \eqref{PDSdeddi}. Moreover, every Carath\'eodo-\\ry solution of \eqref{PDSdeddi} asymptotically converges to this set.

\begin{theorem}
A point $x^\star \in \mathcal{X}$ is a stationary point of \eqref{eq:main_problem} if and only if it is an equilibrium point of \eqref{PDSdeddi}. Moreover, every Carath\'eodory solution of \eqref{PDSdeddi} asymptotically converges to the set of stationary points of \eqref{eq:main_problem}.
\label{stability them}
\end{theorem}

\begin{proof}
We prove both the directions and subsequently the convergence result.

\textbf{Step 1: Equivalence of Stationary Point and Equilibrium Point}
\medskip
\noindent

\textbf{(Stationary point $\Rightarrow$ Equilibrium point).}
Let $x^\ast \in \mathcal{X}$ be a stationary point of \eqref{eq:main_problem}. Then there exists $g(x^\ast) \in \partial f(x^\ast)$ such that
$
\langle g(x^\ast), v \rangle \ge 0,
\quad \forall\, v \in \mathcal{T}_{\mathcal{X}}(x^\ast).
$
We show that
\[
0 \in 
\mathcal{P}^{x^\ast}_{\mathcal{T}_{\mathcal{X}}(x^\ast)}
\big(
- \nabla^2 R(x^\ast)^{-1} \partial f(x^\ast)
\big).
\]

Let
$
\nu_1 =
\mathcal{P}^{x^\ast}_{\mathcal{T}_{\mathcal{X}}(x^\ast)}
\big(
- \nabla^2 R(x^\ast)^{-1} g(x^\ast)
\big).
$ 
By definition of the projection,\\
$
\nu_1
=
\argmin_{\nu \in \mathcal{T}_{\mathcal{X}}(x^\ast)}
\left\|
\nu + \nabla^2 R(x^\ast)^{-1} g(x^\ast)
\right\|_{x^\ast}^{2},
$
where $\|z\|_{x^\ast}^2 := \langle \nabla^2 R(x^\ast) z, z \rangle$. 
The first-order optimality condition yields
\[
\left\langle
\nabla^2 R(x^\ast)
\Big(
\nu_1 + \nabla^2 R(x^\ast)^{-1} g(x^\ast)
\Big),
\nu - \nu_1
\right\rangle
\ge 0,
\quad \forall\, \nu \in \mathcal{T}_{\mathcal{X}}(x^\ast).
\]

Since $\mathcal{T}_{\mathcal{X}}(x^\ast)$ is a cone, choosing $\nu = 2\nu_1$ and $\nu = \tfrac{1}{2}\nu_1$ gives
\[
\left\langle
\nabla^2 R(x^\ast)
\Big(
\nu_1 + \nabla^2 R(x^\ast)^{-1} g(x^\ast)
\Big),
\nu_1
\right\rangle
= 0.
\]
Expanding, we obtain
$
\langle g(x^\ast), \nu_1 \rangle
=
- \langle \nabla^2 R(x^\ast)\nu_1, \nu_1 \rangle.
$ 
Since $\nabla^2 R(x^\ast)$ is positive definite, the right-hand side is strictly negative whenever $\nu_1 \neq 0$, which contradicts the stationarity condition. Hence $\nu_1 = 0$, proving that $x^\ast$ is an equilibrium point.

\medskip
\noindent
\textbf{(Equilibrium point $\Rightarrow$ Stationary point).}
Let $x^\ast \in \mathcal{X}$ be an equilibrium point of \eqref{PDSdeddi}. Then there exists $g(x^\ast) \in \partial f(x^\ast)$ such that
$
\mathcal{P}^{x^\ast}_{\mathcal{T}_{\mathcal{X}}(x^\ast)}
\big(
- \nabla^2 R(x^\ast)^{-1} g(x^\ast)
\big)
= 0.
$ 
Thus, $\nu = 0$ is the solution to 
$
\min_{\nu \in \mathcal{T}_{\mathcal{X}}(x^\ast)}
\left\|
\nu + \nabla^2 R(x^\ast)^{-1} g(x^\ast)
\right\|_{x^\ast}^{2}.
$ 
The first-order optimality condition at $\nu = 0$ yields
$
\left\langle
g(x^\ast), \nu
\right\rangle
\ge 0,
\quad \forall\, \nu \in \mathcal{T}_{\mathcal{X}}(x^\ast),
$ 
which is precisely the stationarity condition.

\medskip
\noindent
\textbf{Step 2: Stability and convergence.}

Consider the Lyapunov function $V(x) := f(x) - f^\ast.$ 
Since $x(\cdot)$ is a Carath\'eodory solution of \eqref{PDSdeddi}, it is absolutely continuous, and hence $f(x(t))$ is absolutely continuous and differentiable almost everywhere.

The Lie derivative of $V$ along \eqref{PDSdeddi} is given by
\[
\mathcal{L}V(x)
=
\left\{
\langle \nabla f(x), \nu \rangle
\;\middle|\;
\nu =
\mathcal{P}^{x}_{\mathcal{T}_{\mathcal{X}}(x)}
\big(
- \nabla^2 R(x)^{-1} g(x)
\big), \; g(x) \in \partial f(x)
\right\}.
\]

Let
${\displaystyle
\nu =
\mathcal{P}^{x}_{\mathcal{T}_{\mathcal{X}}(x)}
\big(
- \nabla^2 R(x)^{-1} g(x)
\big).
}$ 
From the optimality condition, we obtain
\[
\langle g(x), \nu \rangle
=
- \langle \nabla^2 R(x)\nu, \nu \rangle.
\]
Since $\nabla^2 R(x)$ is positive definite,
$
\langle \nabla f(x), \nu \rangle \le 0.
$ 
Thus,
\[
\frac{d}{dt} f(x(t)) \le 0
\quad \text{for almost every } t.
\]

Define the sublevel set
$
\Omega_c := \{ x \in \mathcal{X} \mid f(x) \le c \},
$ 
where $c = f(x_0)$. Then $\Omega_c$ is compact and positively invariant.

By the invariance principle (Theorem~2 in \cite{cortes2008discontinuous}), every Carath\'eodory solution converges to the largest invariant set contained in
${\displaystyle 
\Omega_c \cap
\overline{
\left\{
y \in \mathcal{X}
\;\middle|\;
0 \in \mathcal{L}V(y)
\right\}
}.
}$ 
From the identity\\
$
\langle g(y), \nu \rangle
=
- \langle \nabla^2 R(y)\nu, \nu \rangle,
$ 
we deduce that $0 \in \mathcal{L}V(y)$ which implies $\nu = 0$. Hence
\[
0 =
\mathcal{P}^{y}_{\mathcal{T}_{\mathcal{X}}(y)}
\big(
- \nabla^2 R(y)^{-1} g(y)
\big),
\]
which shows that $y$ is a stationary point. 
Therefore, every trajectory converges to the set of stationary points, completing the proof.
\end{proof}

\section{Finite-Time Analysis: Convex Objectives}
We now move beyond asymptotic guaranties and study the finite-time behavior of stochastic mirror descent through explicit probabilistic bounds. The analysis differs for convex and non-convex settings, which we treat separately. In the convex case, we bound the probability that the iterates remain within a prescribed neighborhood of the optimal solution, leading to the notion of sample complexity introduced below.

\begin{definition}[Sample Complexity]
\label{def:sample_complexity}
Let $\epsilon > 0$ and $p \in (0,1)$. For a sequence $\{z_n\}_{n \geq 1}$, the sample complexity $N(\epsilon, p)$ is the smallest positive integer such that for all $n \geq N(\epsilon, p)$,
\begin{equation}
    \mathbb{P}\big(f(z_n) - f^\ast \geq \epsilon\big) \leq p.
\end{equation}
This corresponds to the number of iterations required to obtain an $\epsilon$-optimal solution with confidence at least $1-p$.
\end{definition}
Throughout this section, we impose smoothness assumptions on the objective function.
We begin by stating the following assumptions, which will play a central role in the analysis.

\begin{assumption}
The function $f$ is continuously differentiable, and its gradient is \\ $L$-Lipschitz continuous. That is, for all $x,y \in \mathcal{X}$,
$
\norm{\nabla f(x) - \nabla f(y)}_\ast \leq L \norm{x - y}.
$
\label{ass8}
\end{assumption}

\begin{assumption}
There exists a Borel measurable function 
$
\widetilde{G} : \mathcal{X} \times \mathcal{S} \to \mathbb{R}^d
$ 
such that the following Poisson equation holds:
\begin{equation}
G(x,s) - \int_{\mathcal{S}} G(x,s') \, \mu_x(ds')
=
\widetilde{G}(x,s) - (\Pi \widetilde{G})(x,s),
\end{equation}
where
${\displaystyle 
(\Pi \widetilde{G})(x,s)
=
\int_{\mathcal{S}} \widetilde{G}(x,s') \, \Pi(x,s)(ds').
}$ 
Further, 
${\displaystyle
\sup_{x \in \mathcal{X},\, s \in \mathcal{S}} \norm{\widetilde{G}(x,s)}_\ast \leq \mathbf{G}.
}$
\label{ass:Poisson}
\end{assumption}
Recall from Assumption~\ref{assumpton3} that $\nabla f(x) = \int_{\mathcal{S}} G(x,s') \, \mu_x(ds').$
Moreover, from the boundedness of $\widetilde{G}$, we obtain
\begin{equation*}
\begin{aligned}
\norm{(\Pi \widetilde{G})(x,s)}_\ast
&= \norm{\int_{\mathcal{S}} \widetilde{G}(x,s') \, \Pi(x,s)(ds')}_\ast 
\le \int_{\mathcal{S}} \norm{\widetilde{G}(x,s')}_\ast \, \Pi(x,s)(ds') \le \mathbf{G}. 
\end{aligned}
\end{equation*}

\begin{assumption}
There exists a constant $L_{\Pi} > 0$ such that for all $x_1, x_2 \in \mathcal{X}$ and all $s \in \mathcal{S}$,
$
\|\widetilde{G}(x_1,s) - \widetilde{G}(x_2,s)\|_\ast
\le
L_{\Pi} \|x_1 - x_2\|.
$
That is, the mapping $\widetilde{G}(x,s)$ is Lipschitz continuous with respect to $x$, uniformly in $s$.
\label{assumptionLipschitzGtilde}
\end{assumption}

\begin{remark}
This assumption is natural in our setting. The function $\widetilde{G}$ arises from the Poisson equation associated with the Markov kernel, and under standard conditions—such as Lipschitz continuity of $G(x,s)$ in $x$ and uniform ergodicity of the chain—it inherits Lipschitz continuity in $x$. Such properties are classical in the analysis of Markov processes and stochastic approximation \cite{doucet2017asymptotic}.
\end{remark}

We now present a key proposition showing that the Markov noise in stochastic mirror descent can be decomposed into a martingale difference term and a bias term that vanishes asymptotically.

\begin{theorem}
Let $\{x_n\}$ be the sequence generated by the stochastic mirror descent algorithm \eqref{mirror_descent1}, and suppose Assumptions~\ref{ass:Poisson} and \ref{assumptionLipschitzGtilde} hold. 
Define the noise sequence
\[
S_{n+1} := G(x_n,S_n) - \nabla f(x_n) = A_{1,n+1} + A_{2,n+1}, \; \; \text{where,}
\]
\begin{itemize}
\item $A_{1,n+1} = \widetilde{G}(x_n,S_{n+1}) - (\Pi \widetilde{G})(x_n,S_n)$, $n \geq 0$, is a martingale difference sequence with respect to the filtration  $\mathcal{F}_n = \sigma (x_k, S_k \; | \; 0 \leq k \leq n )$, $n\geq 0$, 
i.e.,\\ $\mathbb{E}[A_{1,n+1} \mid \mathcal{F}_n] = 0$ for all $n$. 
\item The remainder term $A_{2,n+1} = \widetilde{G}(x_n,S_n)
-
\widetilde{G}(x_n,S_{n+1})$ satisfies
\begin{equation}
    \left|
    \sum_{k=1}^n \alpha_k \langle x^\ast - x_k, A_{2,k+1} \rangle
    \right|
    \leq
    4 \mathbf{G} D
    + \frac{\mathbf{G}^2}{2 \sigma_R} \sum_{k=1}^n \alpha_k^2
    + \sum_{k=1}^n \frac{D L_\Pi \mathbf{G}}{\sigma_R} \alpha_k^2.
    \label{innerPoisson}
\end{equation}
\end{itemize}
\label{theorem8}
\end{theorem} 
\begin{proof}
The proof begins by expressing the stochastic gradient noise through the Poisson equation. We then decompose the noise into a martingale difference term and a remainder term, and estimate these two parts separately.
Recall that the mean field is given by
${\displaystyle
\nabla f(x_n)
=
\int_{\mathcal{S}} G(x_n,s')\,\mu_{x_n}(ds'),
}$ 
which follows directly from Assumption~\ref{assumpton3}. Therefore, $S_{n+1}$ measures the deviation of the stochastic observation $G(x_n,S_n)$ from its stationary average.
Using the Poisson equation from Assumption~\ref{ass:Poisson}, we obtain
\begin{equation*}
\begin{aligned}
S_{n+1}
&= G(x_n,S_n) - \nabla f(x_n) = \widetilde{G}(x_n,S_n) - (\Pi \widetilde{G})(x_n,S_n).
\end{aligned}
\end{equation*}
We now introduce the decomposition
\begin{equation*}
\begin{aligned}
S_{n+1}
&=
\underbrace{\widetilde{G}(x_n,S_{n+1})
-
(\Pi \widetilde{G})(x_n,S_n)}_{A_{1,n+1}}
 +
\underbrace{\widetilde{G}(x_n,S_n)
-
\widetilde{G}(x_n,S_{n+1})}_{A_{2,n+1}} .
\end{aligned}
\end{equation*}
\textbf{Step 1: $A_{1,n+1}$ is a martingale difference term.}

Since $S_{n+1}$ is generated according to the transition kernel $\Pi(x_n,S_n)$, we have
\begin{equation*}
\mathbb{E}\!\left[
\widetilde{G}(x_n,S_{n+1})
\mid
\mathcal{F}_n
\right]
=
\int_{\mathcal{S}} \widetilde{G}(x_n,s')\,\Pi(x_n,S_n)(ds').
\end{equation*}
By definition of the operator $\Pi$, the right-hand side equals $(\Pi \widetilde{G})(x_n,S_n).$
Consequently,
\begin{equation*}
\mathbb{E}[A_{1,n+1}\mid \mathcal{F}_n] = 0.
\end{equation*}
\textbf{ Step - 2: Proof of \eqref{innerPoisson}.}
We first write
\begin{equation*}
\begin{aligned}
A_{2,n+1}
&=
\widetilde{G}(x_n,S_n)
-
\widetilde{G}(x_n,S_{n+1})
\\
&=
\underbrace{\widetilde{G}(x_n,S_n)
-
\widetilde{G}(x_{n+1},S_{n+1})}_{T_{1,n+1}}
+
\underbrace{\widetilde{G}(x_{n+1},S_{n+1})
-
\widetilde{G}(x_n,S_{n+1})}_{T_{2,n+1}}.
\end{aligned}
\end{equation*}
Define $\gamma_k := \widetilde{G}(x_k,S_k),
\qquad
b_k := \alpha_k (x^\ast - x_k).$
Then, considering the first term, we have
\begin{equation}
\begin{aligned}
& \left|
\sum_{k=1}^{n} \alpha_k
\langle x^\ast - x_k,\,
\widetilde{G}(x_k,S_k)-\widetilde{G}(x_{k+1},S_{k+1}) \rangle
\right|
\\ \leq &
\left|
\sum_{k=1}^{n}
\langle b_k,\gamma_k-\gamma_{k+1}\rangle
\right|
\leq 
|\langle b_1,\gamma_1\rangle|
+
|\langle b_n,\gamma_{n+1}\rangle|
+
\underbrace{
\left|
\sum_{k=1}^{n-1}
\langle b_{k+1}-b_k,\gamma_{k+1}\rangle
\right|
}_{\mathbf{T}}.
\end{aligned}
\label{Ry}
\end{equation}

Now we estimate the term $\mathbf{T}$. Using the boundedness of $\gamma_{k+1}$, we get
\begin{equation*}
\begin{aligned}
&\left|
\sum_{k=1}^{n-1}
\langle b_{k+1}-b_k,\gamma_{k+1}\rangle
\right|
\\ &\leq
\mathbf{G}
\sum_{k=1}^{n-1}
\|b_{k+1}-b_k\|
\overset{(a)}{\leq}
\mathbf{G}
\sum_{k=1}^{n-1}
\left\|
\alpha_{k+1}(x^\ast-x_{k+1})
-
\alpha_k(x^\ast-x_k)
\right\|
\\
&\overset{(b)}{\leq}
\mathbf{G}
\sum_{k=1}^{n-1}
\left(
\|x^\ast\|(\alpha_k-\alpha_{k+1})
+
\alpha_k
\left\|
x_k-\frac{\alpha_{k+1}}{\alpha_k}x_{k+1}
\right\|
\right)
\\
&\overset{(c)}{\leq}
\mathbf{G}
\sum_{k=1}^{n-1}
\left(
\|x^\ast\|(\alpha_k-\alpha_{k+1})
+
\alpha_k\|x_k-x_{k+1}\|
+
\alpha_k\left\|
x_{k+1}-\frac{\alpha_{k+1}}{\alpha_k}x_{k+1}
\right\|
\right)
\\
&\overset{(d)}{\leq}
\mathbf{G}D\alpha_1
+
\frac{\mathbf{G}^2}{2\sigma_R}
\sum_{k=1}^{n-1}\alpha_k^2
+
\mathbf{G}D\alpha_1.
\end{aligned}
\end{equation*}

The inequality in \((c)\) follows from the triangle inequality, and the inequality in \((d)\) follows from \eqref{eqna} and the fact that the diameter of the constraint set is $D$. Therefore, from \eqref{Ry}, we obtain
\begin{equation*}
\begin{aligned}
\left|
\sum_{k=1}^{n} \alpha_k
\langle x^\ast - x_k,\,
\widetilde{G}(x_k,S_k)-\widetilde{G}(x_{k+1},S_{k+1}) \rangle
\right|
\leq
4 \mathbf{G}D
+
\frac{\mathbf{G}^2}{2\sigma_R}
\sum_{k=1}^{n}\alpha_k^2.
\end{aligned}
\end{equation*}
Without loss of generality, we assume here that $\alpha_1\leq 1$. 
Next, consider the term $T_{2,n+1}$. Then
\begin{equation*}
\begin{aligned}
&\left|
\sum_{k=1}^{n} \alpha_k
\langle x^\ast - x_k,\,
\widetilde{G}(x_k,S_{k+1})-\widetilde{G}(x_{k+1},S_{k+1}) \rangle
\right|
\\ &\leq
\sum_{k=1}^{n}
\alpha_k
\|x^\ast-x_k\|
\,
\|\widetilde{G}(x_k,S_{k+1})-\widetilde{G}(x_{k+1},S_{k+1})\|_\ast
\\
&\leq
\sum_{k=1}^{n}
\alpha_k D L_\Pi \|x_{k+1}-x_k\|
\leq
\sum_{k=1}^{n}
\frac{D L_\Pi \mathbf{G}}{\sigma_R}\alpha_k^2.
\end{aligned}
\end{equation*}
The second inequality is because of the diameter of the constrained set $\mathcal{X}$ and in view of Assumption \ref{ass:Poisson}. The final inequality follows from \eqref{eqna}.
\end{proof}

\begin{remark}
Theorem~\ref{theorem8} leverages the Poisson equation to represent the Markov noise, revealing that stochastic mirror descent can be viewed as its classical counterpart with independent noise \cite{rsa}, perturbed by a diminishing bias term.
\end{remark}

\begin{theorem}[Finite-time concentration for the ergodic iterate]
Let $\{x_n\}$ be the sequence generated by the stochastic mirror descent algorithm \eqref{mirror_descent1}, and suppose Assumptions~\ref{ass:Poisson} and  \ref{assumptionLipschitzGtilde} hold. Define the weighted average iterate
${\displaystyle
z_n := \frac{\sum_{k=1}^n \alpha_k x_k}{\sum_{k=1}^n \alpha_k}.
}$ 
Fix $\epsilon>0$, and let $n_1>0$ be such that for all $n \geq n_1$,
\begin{equation*}
\begin{aligned}
    \sum_{k=1}^{n} \alpha_k
    &\geq \max \Bigg\{ \frac{3}{\epsilon}\Big( \mathbb{D}_R(x^\ast,x_1) + 4 \mathbf{G} D \Big), \frac{3\big(\mathbf{G}+\mathbf{G}^2+2L_\Pi D \mathbf{G}\big)}{2\sigma_R\epsilon}
    \sum_{k= 1}^n \alpha_k^2  \Bigg\}.
    \end{aligned}
\end{equation*}
Then, for all $n\ge n_1$,
${\displaystyle
\mathbb{P}\big(f(z_n)-f^\ast \ge \epsilon\big)
\leq
\exp\!\left(
-\frac{\epsilon^2\left(\sum_{k=1}^{n}\alpha_k\right)^2}
{72D^2\mathbf{G}^2\sum_{k=1}^{n}\alpha_k^2}
\right).
}$
\label{cvrgncethm}
\end{theorem}

\begin{proof}
The first-order optimality condition for \eqref{mirror_descent1} gives
\begin{equation}
\begin{aligned}
    \alpha_n \langle G(x_n,S_n), x - x_{n+1} \rangle
    \geq &
    - \langle \nabla R(x_{n+1}) - \nabla R(x_n), x - x_{n+1} \rangle,
    \qquad \forall\, x \in \mathcal{X},
    \\ = & \mathbb{D}_R(x_{n+1},x_n)+ \mathbb{D}_R(x,x_{n+1})-\mathbb{D}_R(x,x_n).
\end{aligned}
\label{firstor}
\end{equation}
Rearranging the terms and choosing $x=x^\ast \in \mathcal{X}$, we obtain
\begin{equation*}
\begin{aligned}
    \mathbb{D}_R(x^\ast,x_{n+1})
   & \leq\,
    \mathbb{D}_R(x^\ast,x_n)
    + \alpha_n \langle G(x_n,S_n), x^\ast - x_n \rangle
   \\ & + \alpha_n \langle G(x_n,S_n), x_n - x_{n+1} \rangle
    - \mathbb{D}_R(x_n,x_{n+1}).
\end{aligned}
\end{equation*}
Now, by Young-Fenchel inequality,
\begin{equation*}
    \alpha_n \langle G(x_n,S_n), x_n-x_{n+1} \rangle
    \leq
    \frac{\alpha_n^2}{2\sigma_R}\norm{G(x_n,S_n)}_\ast^2
    + \frac{\sigma_R}{2}\norm{x_n-x_{n+1}}^2,
\end{equation*}
and, by strong convexity, $ -\mathbb{D}_R(x_n,x_{n+1})
    \leq
    -\frac{\sigma_R}{2}\norm{x_n-x_{n+1}}^2.$
Therefore, from \eqref{firstor}, we get
\begin{equation}
    \mathbb{D}_R(x^\ast,x_{n+1})
    \leq
    \mathbb{D}_R(x^\ast,x_n)
    + \alpha_n \langle G(x_n,S_n), x^\ast - x_n \rangle
    + \frac{\alpha_n^2}{2\sigma_R}\mathbf{G}.
    \label{cvxity}
\end{equation} 
Theorem~\ref{theorem8} gives
$
    G(x_n,S_n) = \nabla f(x_n) + S_{n+1},
$ 
with $S_{n+1}=A_{1,n+1}+A_{2,n+1}$. Hence, using the convexity of $f$, \eqref{cvxity} becomes
\begin{equation*}
\begin{aligned}
    \mathbb{D}_R(x^\ast,x_{n+1})
    \leq\,
    \mathbb{D}_R(x^\ast,x_n)
    - \alpha_n \big(f(x_n)-f^\ast\big)
    + \alpha_n \langle x^\ast-x_n, S_{n+1} \rangle
    + \frac{\alpha_n^2}{2\sigma_R}\mathbf{G}.
\end{aligned}
\end{equation*}
Applying the telescopic sum, we obtain
\begin{equation*}
\begin{aligned}
    \sum_{k=1}^n \alpha_k \big(f(x_k)-f^\ast\big)
    \leq
    \mathbb{D}_R(x^\ast,x_1)
    + \sum_{k=1}^{n} \alpha_k \langle x^\ast-x_k, S_{k+1} \rangle
    + \sum_{k=1}^n \frac{\alpha_k^2}{2\sigma_R}\mathbf{G}.
\end{aligned}
\end{equation*}

Now define
$
z_n = \frac{\sum_{k=1}^n \alpha_k x_k}{\sum_{k=1}^n \alpha_k}.
$
Then, 
\begin{equation}
    f(z_n)-f^\ast
    \leq
    \frac{
    \mathbb{D}_R(x^\ast,x_1)
    + \sum_{k=1}^{n} \alpha_k \langle x^\ast-x_k, S_{k+1} \rangle
    + \sum_{k = 1}^n \frac{\alpha_k^2}{2\sigma_R}\mathbf{G}
    }{
    \sum_{k=1}^n \alpha_k
    }.
    \label{Rumpays}
\end{equation}

Recall that $S_{n+1}=A_{1,n+1}+A_{2,n+1}$, and in view of Theorem \ref{theorem8}, we have
\begin{equation*}
    \left|
    \sum_{k=1}^n \alpha_k \langle x^\ast - x_k, A_{2,k+1} \rangle
    \right|
    \leq
    4 \mathbf{G} D
    + \frac{\mathbf{G}^2}{2 \sigma_R} \sum_{k=1}^n \alpha_k^2
    + \sum_{k=1}^n \frac{D L_\Pi \mathbf{G}}{\sigma_R} \alpha_k^2.
\end{equation*}
Thus, \eqref{Rumpays} can be written as
{\small 
\begin{equation*}
\begin{aligned}
   & f(z_n)-f^\ast
   \\ \leq &
    \frac{
    \mathbb{D}_R(x^\ast,x_1) + 4 \mathbf{G} D
    + \sum_{k=1}^{n} \alpha_k \langle x^\ast-x_k, A_{1,k+1} \rangle
    + \sum\limits_{k = 1}^n \frac{\alpha_k^2}{2\sigma_R}
    \big( \mathbf{G} + \mathbf{G}^2 + 2 D L_\Pi \mathbf{G} \big)
    }{
    \sum_{k=1}^n \alpha_k
    }.
\end{aligned}
\end{equation*}
}
Consider $n_1>0$ such that for all $n \geq n_1$, we have
\begin{equation*}
\begin{aligned}
    \sum_{k=1}^{n} \alpha_k \geq (\frac{3}{\epsilon} \mathbb{D}_R(x^\ast,x_1) +  4 \mathbf{G} D)
    \quad \text{and} \quad
    \sum_{k=1}^{n} \alpha_k \geq
    \frac{3\big(\mathbf{G}+\mathbf{G}^2+2L_\Pi D \mathbf{G}\big)}{2\sigma_R\epsilon}
    \sum_{k = 1}^n \alpha_k^2.
\end{aligned}
\end{equation*}

Now consider the event that $f(z_n)-f^\ast \geq \epsilon$. If this event holds for $n\geq n_1$, then
\begin{equation*}
   \sum_{k=1}^n \alpha_k \langle x^\ast-x_k, A_{1,k+1} \rangle
   \geq
   \frac{\epsilon}{3}\sum_{k=1}^n \alpha_k.
\end{equation*}
In other words, for all $n \geq n_1$,
\begin{equation*}
\begin{aligned}
    \mathbb{P}\big(f(z_n)-f^\ast \geq \epsilon\big)
    \leq
    \mathbb{P}\left(
    \sum_{k=1}^n \alpha_k \langle x^\ast-x_k, A_{1,k+1} \rangle
    \geq
    \frac{\epsilon}{3}\sum_{k=1}^n \alpha_k
    \right).
\end{aligned}
\end{equation*}

The remaining part of the proof is to bound the probability on the right-hand side. For this purpose, define
$
    Z_{n+1}
    :=
    \sum_{k=1}^n \alpha_k \langle x^\ast-x_k, A_{1,k+1} \rangle.
$ 
Then $\{Z_n\}$ is a martingale sequence. We now apply the Azuma-Hoeffding inequality. For this, note that
\begin{equation}
    \begin{aligned}
        \norm{x^\ast - x_k} \leq D,
        \qquad
        \norm{A_{1,k+1}}_\ast \leq 2 \mathbf{G}.
    \end{aligned}
    \label{Azuma}
\end{equation}
Recall that in the proof of Theorem~\ref{theorem8}, we defined
\begin{equation*}
    A_{1,n+1} = \widetilde{G}(x_n,S_{n+1}) - (\Pi \widetilde{G})(x_n,S_n).
\end{equation*}
Thus, the second inequality in \eqref{Azuma} holds in view of Assumption~\ref{ass:Poisson}. Therefore, the martingale increments are bounded as
\begin{equation*}
    \big|Z_{n+1} - Z_n\big|
    =
    \big|\alpha_n \langle x^\ast - x_n, A_{1,n+1} \rangle\big|
    \leq
    \alpha_n \norm{x^\ast - x_n}\,\norm{A_{1,n+1}}_\ast
    \leq
    2 \alpha_n D \mathbf{G}.
\end{equation*}
Applying the Azuma-Hoeffding inequality \cite{vershynin2020high}, we obtain
\begin{equation*}
    \begin{aligned}
        \mathbb{P}\left(
        Z_{n+1} \geq \frac{\epsilon}{3} \sum_{k=1}^{n} \alpha_k
        \right)
        \leq
        \exp\!\left(
        -\frac{\epsilon^2\left(\sum_{k=1}^{n} \alpha_k\right)^2}
        {72 D^2 \mathbf{G}^2 \sum_{k=1}^{n} \alpha_k^2}
        \right).
    \end{aligned}
\end{equation*}
\end{proof} 
In the next Theorem we derive the sample complexity result.  

\begin{theorem}
Fix an accuracy level $\epsilon > 0$ and a probability level $p \in (0,1)$, and let $\alpha_n = \frac{1}{\sqrt{n}}, n\geq 1$.
Let $D_1 := \sup_{x \in \mathcal{X}} \mathbb{D}_R(x^\ast,x)$ and 
${\displaystyle
C := \frac{3\big(\mathbf{G}+\mathbf{G}^2+2L_\Pi D \mathbf{G}\big)}{2\sigma_R}.
}$ 
Define $n_1$ such that, for all $n \geq n_1$,
\begin{equation*}
 n \geq \max\left\{\frac{9}{2 \epsilon^2} \big(D_1^2 + 16 \mathbf{G}^2 D^2\big), \frac{C^2}{4a^2 \epsilon^2} \big(\ln (1+n)\big)^2, \frac{18 D^2 \mathbf{G}^2}{\epsilon^2} \ln\!\left(\frac{1}{p}\right) \ln (1+n)\right\}.
\end{equation*}
Then, for all $n \geq n_1$,
$
 \mathbb{P}\big(f(z_n) - f^\ast \geq \epsilon\big) \leq p.
$
\end{theorem}

\begin{proof}
We begin by recalling two useful bounds obtained from the integration test. For any \( n \in \mathbb{N} \), we have
\[
\sum_{k=1}^{n} \frac{a}{\sqrt{k}} \geq 2a \sqrt{n},
\qquad
\sum_{k=1}^{n} \frac{a^2}{k} \leq a^2 \ln(1+n).
\]

Next, we choose \( n_1 \in \mathbb{N} \) such that for all \( n \geq n_1 \),
\[
2a\sqrt{n} \geq \sum_{k=1}^{n} \alpha_k
\;\geq\;
\frac{3}{\epsilon}\Big( \mathbb{D}_R(x^\ast,x_1) + 4 \mathbf{G} D \Big).
\]

Squaring both sides, and simplifying gives  
${\displaystyle
n
\geq 
\frac{9}{2 a^2 \epsilon^2}
\Big( \mathbb{D}_R(x^\ast,x_1)^2 + 16 \mathbf{G}^2 D^2 \Big)}$ $\stackrel{\triangle}{=} n_1.
$ 
Next, we fix \( n_2 \in \mathbb{N} \) such that for all \( n \geq n_2 \),
$
2a\sqrt{n} \;\geq\; C \ln(1+n),
$ 
where the constant \( C \) is defined in the statement of the theorem. 
Rewriting this condition, we obtain
${\displaystyle 
n
\;\geq\;
\frac{C^2}{4a^2} \, \ln^2(1+n).
}$ 
Thus, choosing 
${\displaystyle 
n_2 = \frac{C^2}{4a^2} \, \ln^2(1+n_2),
}$ 
we ensure that for all \( n \geq n_2 \),
\[
\sum_{k=1}^{n} \alpha_k
\;\geq\;
\frac{3\big(\mathbf{G}+\mathbf{G}^2+2L_\Pi D \mathbf{G}\big)}{2\sigma_R\epsilon}
\sum_{k=1}^{n} \alpha_k^2.
\]

Now, applying Theorem~\ref{cvrgncethm}, we obtain that for all
$
n \;\geq\; \max\{n_1,n_2\},
$
the following probability bound holds:
\[
\mathbb{P}\big(f(z_n)-f^\ast \ge \epsilon\big)
\leq
\exp\!\left(
-\frac{\epsilon^2\left(\sum_{k=1}^{n}\alpha_k\right)^2}
{72D^2\mathbf{G}^2\sum_{k=1}^{n}\alpha_k^2}
\right).
\]

Finally, we determine \( n_3 \in \mathbb{N} \) such that for all \( n \geq n_3 \),
\[
\exp\!\left(
-\frac{\epsilon^2\left(\sum_{k=1}^{n}\alpha_k\right)^2}
{72D^2\mathbf{G}^2\sum_{k=1}^{n}\alpha_k^2}
\right)
\;\leq\; p.
\]

Taking logarithms and simplifying, we obtain the sufficient condition
\[
n
\;\geq\;
\ln\!\left(\frac{1}{p}\right)
\, 18 D^2 \mathbf{G}^2 \ln(1+n).
\]

This completes the proof.
\end{proof}
The following Corollary is now immediate.
\begin{corollary}[Sample Complexity Bound]\label{10}
\label{cor:sample_complexity_bound}
Suppose the conditions of Theorem~\ref{cvrgncethm} hold with the stepsize $\alpha_n = \frac{1}{\sqrt{n}}$. Let $\epsilon > 0$ be the target accuracy and $p \in (0,1)$ be the target failure probability. Then, the total number of iterations $N$ required to guarantee $\mathbb{P}\big(f(z_N) - f^\ast \geq \epsilon\big) \leq p$ is bounded by
\begin{equation}
    N = \tilde{\mathcal{O}}\left( \frac{1}{\epsilon^2} \max\left\{ D_1^2 + \mathbf{G}^2 D^2, \, \frac{\mathbf{G}^4 + \mathbf{G}^2(1+2L_\Pi D)^2}{a^2 \sigma_R^2}, \, \mathbf{G}^2 D^2 \ln\!\left(\frac{1}{p}\right) \right\} \right).
\end{equation}
\label{corollaty-smaple}
\end{corollary}
\begin{remark}
The concentration bound in Corollary~\ref{corollaty-smaple} aligns with known behavior \cite{rsa}. For stochastic mirror descent, $\mathbf{G} = \sup_{x,s}\|G(x,s)\|_\ast$, and with the KL divergence on the probability simplex, the dual norm is the $\infty$-norm, making $\mathbf{G}$ independent of dimension. In contrast, for Euclidean SGD, the dual norm scales as $\|\cdot\|_2 \sim \sqrt{d}$.
Moreover, starting from $x_0 = \frac{1}{d}\mathbf{1}$, the Bregman divergence satisfies $\sup_{y\in\mathcal{X}} \mathbb{D}_R(x_0,y) = \mathcal{O}(\ln d)$, whereas the Euclidean distance yields $\sup_{y\in\mathcal{X}} \|x_0-y\|^2 = \mathcal{O}(1)$. 
Consequently, on the probability simplex, Euclidean SGD leads to sample complexity $N(\epsilon,p)=\mathcal{O}\big(\tfrac{1}{\epsilon^2} d^2 \ln(\tfrac{1}{p})\big)$, while with KL geometry this improves to $N(\epsilon,p)=\mathcal{O}\big(\tfrac{1}{\epsilon^2} (\ln d) \ln(\tfrac{1}{p})\big)$.
\end{remark}

\section{Finite-Time Analysis: Non-Convex Objectives on the Probability Simplex}

In this section, we revisit \eqref{eq:main_problem1} under the assumption that the constraint set $\mathcal{X}$ is the probability simplex. 
Our primary focus is solving the following optimization problem \eqref{eq:main_problem},
where the constraint set $\mathcal{X}$ is the probability simplex, defined as
\begin{equation*}
    \mathcal{X} = \left\{ x \in \mathbb{R}^d \;\middle|\; \sum_{i=1}^d x[i] = 1, \quad x[i] \geq 0 \right\}.
\end{equation*}

 This constraint set naturally arises in several applications where the decision variable represents a probability distribution, for example in reinforcement learning, where the action space is itself a probability distribution \cite{agarwal2021theory,xiao2022convergence,li2025stochastic}. Consequently, any feasible point must lie in a non-Euclidean geometry, which motivates the use of mirror descent–type methods.
We assume that the objective function $f$ is possibly non-convex and satisfies Assumptions~\ref{assumpton3}, \ref{ass8}, \ref{ass:Poisson}, and \ref{assumptionLipschitzGtilde} with respect to the $1$-norm and its corresponding dual norm, namely the $\infty$-norm.
\medskip

\noindent
Since the domain is the probability simplex, a natural and standard choice of the mirror map is the negative entropy function given by $ R(x) = \sum_{i=1}^d x[i] \ln x[i].$

This choice is particularly convenient because it respects the geometry of the simplex. The associated Bregman divergence is  computed as
\begin{equation*}
    \mathbb{D}_R(x, y) = \sum_{i=1}^d x[i] \ln \frac{x[i]}{y[i]}.
\end{equation*}

\medskip

\noindent
With this choice of mirror map, the mirror descent updates admit a particularly simple and intuitive form. In fact, it can be verified (see \cite{bubeck2011}) that the iteration \eqref{mirror_descent1} can be written component-wise as
\begin{equation*}
    x_{n+1}[i] = \frac{x_n[i] \exp\big(- \alpha_n g(x_n,S_n)[i]\big)}{\sum\limits_{j=1}^d x_n[j] \exp\big(- \alpha_n g(x_n,S_n)[j]\big)}.
\end{equation*}

\medskip

\noindent
An important structural property of this update (see \cite{bubeck2011}) is the following. If the initial point satisfies $x_1 \in \mathrm{reint}(\mathcal{X})$, then all subsequent iterates remain strictly in the relative interior, that is, $x_n \in \mathrm{reint}(\mathcal{X}) \quad \text{for all } n.$
This observation is particularly useful because it allows us to avoid boundary-related complications. 
We obtain the following two important consequences of the above observations.

First, since we are away from the relative boundary, the tangent cone at any point behaves like a half space. As a result, the non-Euclidean projection onto the tangent cone, defined in \eqref{PDSdeddi}, becomes a linear operator. That is, for any vectors $u,v \in \mathbb{R}^d$ and scalar $\alpha$,
\begin{equation*}
    \mathcal{P}_{\mathcal{T}_\mathcal{X}(x)}^x (\alpha u + v)  
    = \alpha \mathcal{P}_{\mathcal{T}_\mathcal{X}(x)}^x (u) 
    + \mathcal{P}_{\mathcal{T}_\mathcal{X}(x)}^x (v).
\end{equation*}

\medskip

\noindent
Second, the tangent cone itself becomes independent of the point $x$. Indeed, for every $x \in \mathrm{reint}(\mathcal{X})$, the tangent cone is given by
$
    \mathcal{T}_\mathcal{X}(x) = \left\{ \nu \in \mathbb{R}^d \;\middle|\; \mathrm{1}^\top \nu = 0 \right\}.
$
\begin{lemma}
Let $x,y \in \mathrm{reint}(\mathcal{X})$. For each point, define
\begin{equation*}
    \nu_x = \mathcal{P}^{x}_{\mathcal{T}_{\mathcal{X}}(x)}
    \Big(
    -\nabla^2 R(x)^{-1} \nabla f(x)
    \Big),
    \qquad
    \nu_y = \mathcal{P}^{y}_{\mathcal{T}_{\mathcal{X}}(y)}
    \Big(
    -\nabla^2 R(y)^{-1} \nabla f(y)
    \Big).
\end{equation*}
Then, the mapping $x \mapsto \nu_x$ is Lipschitz continuous in the $\ell_1$-norm. In particular, we have
\begin{equation*}
    \norm{\nu_x-\nu_y}_1 \leq L_\nu \norm{x-y}_1,
\end{equation*}
where $L_\nu = 2 L + 3 \mathbf{G}$.
\label{Lemma1}
\end{lemma}
\begin{proof}
  Consider
\begin{equation}
\begin{aligned}
    \nu_x 
    &= \mathcal{P}^{x}_{\mathcal{T}_{\mathcal{X}}(x)}
    \Big(
    -\nabla^2 R(x)^{-1} \nabla f(x)
    \Big) \\
    &= \argmin\limits_{\nu \in \mathcal{T}_\mathcal{X}(x)} 
    (\nu+ \nabla^2 R(x)^{-1} \nabla f(x))^\top 
    \nabla^2 R(x) 
    (\nu+ \nabla^2 R(x)^{-1} \nabla f(x)).
\end{aligned}
\label{simopt}
\end{equation}

Recall that
$
\mathcal{T}_\mathcal{X}(x) = \left\{ \nu \in \mathbb{R}^d \;\middle|\; \mathrm{1}^\top \nu = 0 \right\}.
$ 
We define the Lagrangian as
\begin{equation*}
    L(\nu,\lambda)  
    = (\nu+ \nabla^2 R(x)^{-1} \nabla f(x))^\top 
    \nabla^2 R(x) 
    (\nu+ \nabla^2 R(x)^{-1} \nabla f(x))  
    + \lambda \, \mathrm{1}^\top \nu.
\end{equation*}

Since the objective is strongly convex, $\nu_x$ is the unique solution if and only if
\begin{equation*}
    \frac{\partial L}{\partial \nu} = 0, 
    \qquad 
    \frac{\partial L}{\partial \lambda} = 0.
\end{equation*}
By solving the above two algebraic equations 
$
    \nu_x 
    = - \mathrm{diag}(x) \nabla f(x) 
    + \langle x, \nabla f(x) \rangle \, x.
$
In a similar manner, for any $y \in \mathrm{reint}(\mathcal{X})$, we have
$
    \nu_y 
    = - \mathrm{diag}(y) \nabla f(y) 
    + \langle y, \nabla f(y) \rangle \, y.
$ 
We decompose the difference as
\begin{equation}
        \nu_x - \nu_y 
        = \underbrace{- \mathrm{diag}(x) \nabla f(x)  +  \mathrm{diag}(y) \nabla f(y)}_{\mathbf{T}_1}
    + \underbrace{\langle x, \nabla f(x) \rangle \, x -   \langle y, \nabla f(y) \rangle \, y}_{\mathbf{T}_2}.
    \label{nux}
\end{equation}

\textbf{Bound on $\mathbf{T}_1$}
\begin{equation*}
\begin{aligned}
    & \norm{- \mathrm{diag}(x) \nabla f(x)  +  \mathrm{diag}(y) \nabla f(y)}_1 \\
    \leq {} & \norm{- \mathrm{diag}(x) \nabla f(x)  +  \mathrm{diag}(x) \nabla f(y)}_1 \\
    &\quad + \norm{\mathrm{diag}(x) \nabla f(y) - \mathrm{diag}(y) \nabla f(y)}_1 \\
    \leq {} & \norm{x}_1 \norm{\nabla f(x) - \nabla f(y)}_\infty 
    + \norm{\nabla f(y)}_\infty \norm{x-y}_1 \\
    \leq {} & L \norm{x-y}_1 + \mathbf{G} \norm{x-y}_1.
\end{aligned}
\end{equation*}

\textbf{Bound on $\mathbf{T}_2$}
\begin{equation*} 
\begin{aligned}
& \norm{\langle x, \nabla f(x) \rangle \, x - \langle y, \nabla f(y) \rangle \, y}_1 \leq \underbrace{\norm{\langle x, \nabla f(x) \rangle \, x - \langle x, \nabla f(y) \rangle \, x}_1}_{\mathbf{T}_3} \\ & \quad \quad + \underbrace{\norm{ \langle x, \nabla f(y) \rangle \, x - \langle y, \nabla f(y) \rangle \, x}_1}_{\mathbf{T}_4} + \underbrace{\norm{\langle y, \nabla f(y) \rangle \, x - \langle y, \nabla f(y) \rangle \, y}_1}_{\mathbf{T}_5} \end{aligned} 
\end{equation*}

\textbf{Bound on $\mathbf{T}_3$}
\begin{equation*}
\begin{aligned}
    \norm{\langle x, \nabla f(x) \rangle \, x -   \langle x, \nabla f(y) \rangle \, x}_1
    &\leq \norm{x}_1 \norm{\nabla f(x) - \nabla f(y)}_\infty \leq L \norm{x-y}_1.
\end{aligned}
\end{equation*}

\textbf{Bound on $\mathbf{T}_4$}
\begin{equation*}
\begin{aligned}
    \norm{ \langle x, \nabla f(y) \rangle \, x -  \langle y, \nabla f(y) \rangle \, x}_1
    &\leq \norm{x-y}_1 \norm{\nabla f(y)}_\ast \leq \mathbf{G} \norm{x-y}_1.
\end{aligned}
\end{equation*}

\textbf{Bound on $\mathbf{T}_5$}
\begin{equation*}
\begin{aligned}
    \norm{\langle y, \nabla f(y) \rangle \, x - \langle y, \nabla f(y) \rangle \, y}_1
    &\leq \norm{y}_1 \norm{\nabla f(y)}_\infty \norm{x-y}_1 \leq \mathbf{G} \norm{x-y}_1.
\end{aligned}
\end{equation*}
In the foregoing bounds, we make use of the fact that $\|x\|_1 \leq 1$, since ${\cal X}$ is the probability simplex.  
Thus,
${\displaystyle 
    \norm{\nu_x-\nu_y}_1 \leq (2L + 3 \mathbf{G}) \norm{x-y}_1}$, 
completing the proof.
\end{proof}
Before presenting the main theorem of this section, we highlight an important conceptual point regarding the notion of stationarity.

In standard unconstrained Euclidean optimization, the stationarity gap is naturally measured by the Euclidean norm of the gradient, namely $\norm{\nabla f(x)}_2$. However, in our constrained and non-Euclidean setting, this measure is no longer appropriate. There are two key reasons for this. First, the Euclidean gradient does not account for the active constraints imposed by the feasible set $\mathcal{X}$. Second, it ignores the local geometry induced by the mirror map $R$.


To address this issue, we measure stationarity using the local norm induced by the mirror map, applied to the projected vector field. Specifically, we define the stationarity gap as\\
$
\text{Gap}(x) := \|\nu_x\|_x = \sqrt{\nu_x^\top \nabla^2 R(x) \nu_x},
$
where
${\displaystyle 
\nu_x = \mathcal{P}_{\mathcal{T}_\mathcal{X}(x)}^x \big( -\nabla^2 R(x)^{-1} \nabla f(x) \big)
}$.
%
%
This definition has a natural interpretation: the vector $\nu_x$ represents the feasible descent direction after incorporating both the constraints and the underlying geometry. Consequently, the quantity $\|\nu_x\|_x$ measures how far the point $x$ is from the stationarity in a manner consistent with the geometry of the problem.
In particular, $x$ is a stationary point if and only if $\nu_x = 0$ (see Theorem \ref{stability them} for more details).
%
%
Finally, it is worth noting that mirror descent can be interpreted as gradient descent on a Riemann manifold equipped with the metric induced by $\nabla^2 R(x)$. The above notion of a stationarity gap is standard in deterministic optimization on Riemann manifolds \cite{Boumal_2023}.
\begin{theorem}[Non-asymptotic convergence under Markov noise]
\label{thm:main_convergence}

Fix $\epsilon > 0$. Let $n_0$ be sufficiently large, such that for all $n \geq n_0$,
\begin{equation*}
\begin{aligned}
    \sum_{k=1}^{n} \alpha_k
    &\geq \max \Bigg\{ \frac{3}{\epsilon^2}\Big( f(x_1) + 3 \mathbf{G} \Big), \frac{3}{\epsilon^2}  \left( 
    L_\nu \mathbf{G}^2 
    + L_{\Pi} \mathbf{G}^2 
    + \frac{L \mathbf{G}^2}{2} 
    \right) 
    \sum_{k= 1}^n \alpha_k^2  \Bigg\}.
    \end{aligned}
\end{equation*}

Then, for all $n \geq n_0$ the following high-probability bound holds:
\begin{equation*}
    \mathbb{P}\Big(\min_{1 \leq k \leq n} \text{Gap}(x_k) \geq \epsilon \Big)
    \leq 
    \exp \left( 
    - \frac{\epsilon^4 \left(\sum\limits_{k=1}^{n} \alpha_k\right)^2}
    {18 \mathbf{G}^4 \sum\limits_{k=1}^{n} \alpha_k^2 } 
    \right).
\end{equation*}
\label{concentration - nonconvex}
\end{theorem}

\begin{proof}
Since the objective function $f$ is $L$-smooth, we can invoke the standard descent lemma, which gives
\begin{equation}
    f(x_{n+1}) \leq f(x_n) + \inprod{\nabla f(x_n)}{x_{n+1}-x_n} + \frac{L}{2}\norm{x_{n+1}-x_n}^2.
    \label{lsmoothness}
\end{equation}
\noindent
From Corollary \ref{TangentConeRepresentation}, the iterate update can be written as
\begin{equation*}
    x_{n+1}-x_n 
    = \alpha_n \mathcal{P}_{\mathcal{T}_{\mathcal{X}}(x_n)}^{(x_n)}\!\big(-\nabla^2 R(x_n)^{-1}y_n\big) 
    + \alpha_n b_n,
\end{equation*}
where $y_n = G(x_n,S_n)$ and $b_n = o(\alpha_n)$ collects higher-order terms. 
\medskip
\noindent
Next, using Theorem \ref{theorem8}, the stochastic vector field admits the decomposition
$
    y_n = \nabla f(x_n) + A_{1,n+1} + A_{2,n+1}.
$ 
We treat each term separately.
\medskip
\noindent
Since we are working in the relative interior of the simplex, the tangent cone is the linear space
$
\mathcal{T}_{\mathcal{X}}(x_n) = \left\{ \nu \in \mathbb{R}^d \;\middle|\; \mathrm{1}^\top \nu = 0 \right\}.
$ 
As a consequence, the projection onto the tangent cone is a linear operator.
Applying this linearity, we obtain
{\small
\begin{align*}
    &\mathcal{P}_{\mathcal{T}_{\mathcal{X}}(x_n)}^{(x_n)}\!\big(-\nabla^2 R(x_n)^{-1}y_n\big) 
    = \underbrace{\mathcal{P}_{\mathcal{T}_{\mathcal{X}}(x_n)}^{(x_n)}\!\big(-\nabla^2 R(x_n)^{-1}\nabla f(x_n)\big)}_{\nu_n} \\
    &\quad \quad\quad \quad + \underbrace{\mathcal{P}_{\mathcal{T}_{\mathcal{X}}(x_n)}^{(x_n)}\!\big(-\nabla^2 R(x_n)^{-1}A_{1,n+1}\big)}_{\mathbf{T}_{2,n}} 
     + \underbrace{\mathcal{P}_{\mathcal{T}_{\mathcal{X}}(x_n)}^{(x_n)}\!\big(-\nabla^2 R(x_n)^{-1}A_{2,n+1}\big)}_{\mathbf{T}_{3,n}}.
\end{align*}}
\medskip
\noindent
\textbf{Analysis of the stationarity measure:} 
We now analyze the quantity $\nu_n$, which plays the role of the projected descent direction. From the first-order optimality condition, we have
\begin{equation*}
    \Big(\nabla^2 R(x_n)\big(\nu_n + \nabla^2 R(x_n)^{-1}\nabla f(x_n)\big)\Big)^\top 
    (\nu - \nu_n) \geq 0 
    \quad \forall \nu \in \mathcal{T}_{\mathcal{X}}(x_n).
\end{equation*}
First, choosing $\nu = 0$, we obtain
$
    \Big(\nabla^2 R(x_n)\big(\nu_n + \nabla^2 R(x_n)^{-1}\nabla f(x_n)\big)\Big)^\top 
    (-\nu_n) \geq 0.
$ 
Next, choosing $\nu = 2\nu_n$, we obtain
$
    \Big(\nabla^2 R(x_n)\big(\nu_n + \nabla^2 R(x_n)^{-1}\nabla f(x_n)\big)\Big)^\top 
    (\nu_n) \geq 0.
$
Combining the two, we obtain 
$
    \Big(\nabla^2 R(x_n)\big(\nu_n + \nabla^2 R(x_n)^{-1}\nabla f(x_n)\big)\Big)^\top 
    \nu_n = 0.
$ 
Expanding the above expression yields
$
    \nu_n^\top \nabla^2 R(x_n)\nu_n + \nabla f(x_n)^\top \nu_n = 0,
$
which gives
$
    \nabla f(x_n)^\top \nu_n = - \nu_n^\top \nabla^2 R(x_n) \nu_n.
$ 

\medskip
\noindent
\textbf{Telescopic sum and complexity bound}: 
Substituting this relation into the descent inequality \eqref{lsmoothness}, and summing over $k=1$ to $n$, we obtain a telescoping relation.
\medskip
\noindent
More precisely, {\small 
\begin{align*}
    \bigg( \min_{1 \leq k \leq n} \|\nu_k\|_{x_k}^2 \bigg) 
    \sum_{k=1}^n \alpha_k 
    &\leq f(x_1) - \inf f  + \sum_{k=1}^n \alpha_k \inprod{\nabla f(x_k)}{\mathbf{T}_{2,k}} \\
     & \quad + \sum_{k=1}^n \alpha_k \inprod{\nabla f(x_k)}{\mathbf{T}_{3,k}}  + \frac{L \mathbf{G}^2}{2\sigma_R^2} \sum_{k=1}^{n} \alpha_k^2.
\end{align*}}
Next, assume that
${\displaystyle 
\min_{1 \leq k \leq n} \|\nu_k\|_{x_k} \geq \epsilon
}$.
Then, clearly
${\displaystyle 
\min_{1 \leq k \leq n} \|\nu_k\|_{x_k}^2 \geq \epsilon^2
}$. 
Substituting this lower bound into the inequality obtained from the telescoping argument, we obtain
\begin{equation*}
   \epsilon^2 
    \sum_{k=1}^n \alpha_k 
    \leq f(x_1) - \inf f 
    + \sum_{k=1}^n \alpha_k \inprod{\nabla f(x_k)}{\mathbf{T}_{2,k}} 
    + \sum_{k=1}^n \alpha_k \inprod{\nabla f(x_k)}{\mathbf{T}_{3,k}} 
    + \frac{L \mathbf{G}^2}{2\sigma_R^2} \sum_{k=1}^{n} \alpha_k^2.
\end{equation*}

\medskip

\noindent
Next, we use the bound on the Markov bias term $\mathbf{T}_{3,k}$ (with $\alpha_1 \leq 1$), as established in Proposition \ref{ProposMarkov}. Substituting this bound into the above inequality yields
{\small
\begin{align*}
   \epsilon^2 
    \sum_{k=1}^n \alpha_k 
    &\leq f(x_1) + \frac{3 \mathbf{G}^2}{\sigma_R} 
    + \sum_{k=1}^n \alpha_k \inprod{\nabla f(x_k)}{\mathbf{T}_{2,k}}
    + \bigg( \frac{L_\nu \mathbf{G}^2}{\sigma_R} 
    + \frac{L_{\widetilde{G}} \mathbf{G}^2}{\sigma_R} 
    + \frac{L \mathbf{G}^2}{2 \sigma_R^2} \bigg) 
    \sum_{k=1}^{n} \alpha_k^2.
\end{align*}}
Let $n_1 \in \mathbb{N}$ be the smallest integer such that
${\displaystyle 
    f(x_1) + \frac{3 \mathbf{G}^2}{\sigma_R} 
    \leq \frac{\epsilon^2}{3} 
    \sum_{k=1}^{n} \alpha_k.
}$
Similarly, let $n_2 \in \mathbb{N}$ be the smallest integer such that 
${\displaystyle 
    \bigg( \frac{L_\nu \mathbf{G}^2}{\sigma_R} 
    + \frac{L_{\widetilde{G}} \mathbf{G}^2}{\sigma_R} 
    + \frac{L \mathbf{G}^2}{2 \sigma_R^2} \bigg) 
    \sum_{k=1}^n \alpha_k^2 
    \leq \frac{\epsilon^2}{3} 
    \sum_{k=1}^{n} \alpha_k.
}$ 
Then, for all $n \geq n_0 := \max\{n_1,n_2\}$, if the event
${\displaystyle
\left\{ \min_{1\leq k \leq n} \|\nu_k\|_{x_k} \geq \epsilon \right\}
}$
occurs, then 
${\displaystyle
    \sum_{k=1}^{n} \alpha_k \inprod{\nabla f(x_k)}{\mathbf{T}_{2,k}} 
    \geq \frac{\epsilon^2}{3} 
    \sum_{k=1}^{n} \alpha_k.
}$ Consequently, it follows that 
\begin{equation*}
    \mathbb{P}\Big(\min_{1\leq k \leq n} \|\nu_k\|_{x_k} \geq \epsilon \Big) 
    \leq 
    \mathbb{P}\left(
    \sum_{k=1}^{n} \alpha_k 
    \inprod{\nabla f(x_k)}{\mathbf{T}_{2,k}} 
    \geq \frac{\epsilon^2}{3} 
    \sum_{k=1}^{n} \alpha_k 
    \right).
\end{equation*}
\textbf{Analysis of the term $\mathbf{T}_{2,n}$}: 
We now analyze the contribution of the stochastic term $\mathbf{T}_{2,n}$. 
Recall from Theorem \ref{theorem8} that the sequence $\{A_{1,k+1}\}$ forms a martingale difference sequence. Since the projection onto the tangent cone is linear, we have
\begin{align*}
    \mathbb{E}[\mathbf{T}_{2,k} \mid \mathcal{F}_k] 
    &= \mathbb{E}\Big[\mathcal{P}_{\mathcal{T}_{\mathcal{X}}(x_k)}^{(x_k)}\!\big(-\nabla^2 R(x_k)^{-1}A_{1,k+1}\big) \,\Big|\, \mathcal{F}_k\Big] \\ 
    &= \mathcal{P}_{\mathcal{T}_{\mathcal{X}}(x_k)}^{(x_k)}\!\Big( -\nabla^2 R(x_k)^{-1} \mathbb{E}[A_{1,k+1} \mid \mathcal{F}_k ] \Big) \\
    &= 0.
\end{align*}

Consequently, the sequence
$
Z_n := \sum_{k=1}^{n} \alpha_k \inprod{\nabla f(x_k)}{\mathbf{T}_{2,k}}$, $n\geq 1$, 
defines a martingale with respect to the filtration $\{\mathcal{F}_n\}$.
From Assumption \ref{ass:Poisson}, we obtain 
$
\norm{A_{1,k+1}}_\ast \leq \mathbf{G}.
$ 
Next by following similar steps as Lemma \ref{Lemma1}, we obtain 
\begin{equation*}
    T_{2,n} = - \mathrm{diag} (x_n) A_{1,n+1} + \inprod{x_n}{A_{1,n+1}} x_n,
\end{equation*} 
and it is easy to verify that 
${\displaystyle 
    \norm{T_{2,n}} \leq 2 \mathbf{G}
}$. 
Then,
\begin{equation*}
    |Z_k - Z_{k-1}| 
    = \big| \alpha_k \inprod{\nabla f(x_k)}{\mathbf{T}_{2,k}} \big|
    \leq \alpha_k \norm{\nabla f(x_k)}_\ast \norm{\mathbf{T}_{2,k}}
    \leq 2\alpha_k \mathbf{G}^2.
\end{equation*}

Applying the Azuma--Hoeffding inequality, we obtain
\begin{align*}
    \mathbb{P}\Big(\min_{1\leq k \leq n} \text{Gap}(x_k) \geq \epsilon \Big) 
    &\leq \mathbb{P}\left( 
    \sum_{k=1}^{n} \alpha_k \inprod{\nabla f(x_k)}{\mathbf{T}_{2,k}} 
    \geq \frac{\epsilon^2}{3} \sum_{k=1}^{n} \alpha_k 
    \right) \\  
    &\leq \exp \left( 
    - \frac{\epsilon^4 \big(\sum\limits_{k=1}^{n} \alpha_k\big)^2}
    {18 \mathbf{G}^4 \sum\limits_{k=1}^{n} \alpha_k^2 } 
    \right).
\end{align*}
This completes the proof.
\end{proof}
The main idea behind the proof of Theorem \ref{concentration - nonconvex} is to separate the effect of the stochastic noise into two parts as we derive in Theorem \ref{theorem8}. One part is a martingale difference term, which can be controlled using a Azuma-Hoeffding Inequality. The other part is the Markov bias, whose contribution is quantified in Proposition~\ref{prop:T3_bound} which we state below. 

\begin{proposition}[Control of the Markov bias term]
\label{prop:T3_bound}
Under the standing assumptions, the contribution of the Markov bias term $\mathbf{T}_3$ satisfies
\begin{align*}
    \bigg| \sum_{k=1}^{n} \alpha_k \inprod{\nabla f(x_k)}{\mathbf{T}_{3,k}} \bigg|
    \leq 3 \mathbf{G}^2 \alpha_1
    + \bigg(
    L_\nu \mathbf{G}^2
    + L_{\Pi} \mathbf{G}^2
    \bigg)\sum_{k=1}^{n} \alpha_k^2.
\end{align*}
Here,
$
\mathbf{T}_{3,n}
=
\mathcal{P}_{\mathcal{T}_{\mathcal{X}}(x_n)}^{(x_n)}
\!\big(-\nabla^2 R(x_n)^{-1}A_{2,n+1}\big),
$
where $A_{2,n+1}$ is defined in Theorem~\ref{theorem8}.
\label{ProposMarkov}
\end{proposition}
\begin{proof}
  We begin by simplifying the inner product using the self-adjointness of the projection operator under the local metric induced by $\nabla^2 R(x_n)$. In particular, we have
\begin{align*}
   & \inprod{\nabla f(x_n)}{\mathcal{P}_{\mathcal{T}_{\mathcal{X}}(x_n)}^{(x_n)}\!\big(-\nabla^2 R(x_n)^{-1}A_{2,n+1}\big)} 
   \\ = & \inprod{\mathcal{P}_{\mathcal{T}_{\mathcal{X}}(x_n)}^{(x_n)}\!\big(-\nabla^2 R(x_n)^{-1}\nabla f(x_n)\big)}{A_{2,n+1}} = \inprod{\nu_n}{A_{2,n+1}},
\end{align*}
where
${\displaystyle    
\nu_n  = \mathcal{P}_{\mathcal{T}_{\mathcal{X}}(x_n)}^{(x_n)}
    \!\big(-\nabla^2 R(x_n)^{-1}\nabla f(x_n)\big)}$.

From the first-order optimality condition, for all $\nu \in \mathcal{T}_\mathcal{X}(x_n)$,
\begin{equation*}
    \inprod{\nabla^2 R(x_n)\big(\nu_n + \nabla^2 R(x_n)^{-1} \nabla f(x_n)\big)}{\nu - \nu_n} 
    \geq 0.
\end{equation*}

Choosing $\nu = 2\nu_n$ and $\nu = 0$, respectively, we obtain
\begin{equation*}
    \inprod{\nabla^2 R(x_n)\big(\nu_n + \nabla^2 R(x_n)^{-1} \nabla f(x_n)\big)}{\nu_n} = 0.
\end{equation*}

Expanding,  
${\displaystyle 
    \inprod{\nu_n}{\nabla^2 R(x_n)\nu_n}
    = - \inprod{\nu_n}{\nabla f(x_n)}}$.  
Using strong convexity and Hölder’s inequality gives us 
${\displaystyle 
    \|\nu_n\|^2 
    \leq \inprod{\nu_n}{\nabla^2 R(x_n)\nu_n} 
    = - \inprod{\nu_n}{\nabla f(x_n)} 
    \leq \mathbf{G} \|\nu_n\|}$. 
Thus,\\
${\displaystyle 
\|\nu_n\| \leq \mathbf{G}}$. 
Next, we follow a similar approach as with Theorem \ref{theorem8}.  In particular, we decompose
$
A_{2,k+1}
= \widetilde{G}(x_k, S_k) - \widetilde{G}(x_{k+1}, S_{k+1})
+ \widetilde{G}(x_{k+1}, S_{k+1}) - \widetilde{G}(x_k, S_{k+1}).
$


\medskip
\noindent
\textbf{Step 1: Analysis of the telescoping term: }
Let
$
\gamma_k := \widetilde{G}(x_k, S_k)$, 
$b_k := \alpha_k \nu_k.
$ 
Then $\|\gamma_k\|_\ast \leq \mathbf{G}$. Using summation by parts, we get
\begin{align}
    \bigg| \sum_{k=1}^{n} \alpha_k \inprod{\nu_k}{\gamma_k - \gamma_{k+1}} \bigg| 
    &\leq |\inprod{b_1}{\gamma_1}| + |\inprod{b_n}{\gamma_{n+1}}| 
    + \sum_{k=1}^{n-1} |\inprod{b_{k+1}-b_k}{\gamma_{k+1}}| \nonumber \\
    &\leq 2\alpha_1 \mathbf{G}^2 
    + \mathbf{G} \sum_{k=1}^{n-1} \|b_{k+1}-b_k\|. 
    \label{eq:T3_sum_parts}
\end{align}
We now bound $\|b_{k+1}-b_k\|$. Note that 
\begin{align*}
    \|b_{k+1}-b_k\|
    &= \|\alpha_{k+1} \nu_{k+1} - \alpha_k \nu_k\| 
    \leq \|\nu_{k+1}\|(\alpha_k - \alpha_{k+1})
    + \alpha_k \|\nu_{k+1} - \nu_k\|.
\end{align*}
Using $\|\nu_{k+1}\| \leq \mathbf{G}$ and Lipschitz continuity as proven in Lemma \ref{Lemma1},
\begin{align*}
    \|b_{k+1}-b_k\|
    &\leq \mathbf{G}(\alpha_k - \alpha_{k+1})
    + \alpha_k L_\nu \|x_{k+1} - x_k\|.
\end{align*}
Using $\|x_{k+1} - x_k\| \leq \alpha_k \mathbf{G}$ (according to \eqref{eqna}),
\begin{align*}
    \|b_{k+1}-b_k\|
    \leq \mathbf{G}(\alpha_k - \alpha_{k+1})
    + \alpha_k^2 L_\nu \mathbf{G}.
\end{align*}
Substituting,
\begin{align*}
    \bigg| \sum_{k=1}^{n} \alpha_k \inprod{\nu_k}{\gamma_k - \gamma_{k+1}} \bigg| 
    &\leq 2\alpha_1 \mathbf{G}^2
    + \mathbf{G} \sum_{k=1}^{n-1} 
    \left[
    \mathbf{G}(\alpha_k - \alpha_{k+1})
    + \alpha_k^2 L_\nu \mathbf{G}
    \right].
\end{align*}
Since $\sum\limits_{k=1}^{n-1} (\alpha_k - \alpha_{k+1}) = \alpha_1 - \alpha_n$,
\begin{align*}
    \bigg| \sum_{k=1}^{n} \alpha_k \inprod{\nu_k}{\gamma_k - \gamma_{k+1}} \bigg| 
    &\leq 2\alpha_1 \mathbf{G}^2
    + \mathbf{G}^2(\alpha_1 - \alpha_n)
    + L_\nu \mathbf{G}^2 \sum_{k=1}^{n-1} \alpha_k^2 \\
    &\leq 3 \mathbf{G}^2 \alpha_1
    + L_\nu \mathbf{G}^2 \sum_{k=1}^{n} \alpha_k^2.
\end{align*}
\medskip
\noindent
 \textbf{Step 2: Analysis of the Lipschitz difference term}

Let $L_{\Pi}$ denote the Lipschitz constant of $\widetilde{G}(\cdot, S)$. Then,
\begin{align*}
 &   \bigg| \sum_{k=1}^{n} \alpha_k 
    \inprod{\nabla f(x_k)}{\widetilde{G}(x_{k+1},S_{k+1}) - \widetilde{G}(x_k,S_{k+1})} \bigg|
   \\ \leq & \sum_{k=1}^{n} \alpha_k 
    \|\nabla f(x_k)\|_\ast 
    \|\widetilde{G}(x_{k+1},S_{k+1}) - \widetilde{G}(x_k,S_{k+1})\|.
\end{align*}

Using $\|\nabla f(x_k)\|_\ast \leq \mathbf{G}$ and Lipschitz continuity of $\widetilde{G}(\cdot,\cdot)$ in the first argument, uniformly w.r.t.~the second, we obtain 
\begin{align*}
    \bigg| \sum_{k=1}^{n} \alpha_k 
    \inprod{\nabla f(x_k)}{\widetilde{G}(x_{k+1},S_{k+1}) - \widetilde{G}(x_k,S_{k+1})} \bigg|
    &\leq \sum_{k=1}^{n} \alpha_k \mathbf{G} L_{\Pi} \|x_{k+1}-x_k\|.
\end{align*}

Using $\|x_{k+1} - x_k\| \leq \alpha_k \mathbf{G}$,
\begin{align*}
    \bigg| \sum_{k=1}^{n} \alpha_k 
    \inprod{\nabla f(x_k)}{\widetilde{G}(x_{k+1},S_{k+1}) - \widetilde{G}(x_k,S_{k+1})} \bigg|
    &\leq L_{\Pi} \mathbf{G}^2 \sum_{k=1}^{n} \alpha_k^2.
\end{align*}
Combining this with Step~1 completes the proof.
\end{proof}
Before concluding, we present an important consequence of Theorem~\ref{concentration - nonconvex} under an additional structural assumption, commonly known as the relative gradient domination condition (or the Riemannian Polyak--\L{}ojasiewicz condition). This condition leads to stronger convergence guarantees and naturally appears in reinforcement learning problems, particularly in average-reward Markov decision processes \cite{li2025stochastic}.
\begin{assumption}
Let $x^\ast$ be a global minimizer of the problem \eqref{eq:main_problem}, where the constraint set $\mathcal{X}$ is a probability simplex. Assume that the objective function $f$ satisfies the gradient domination condition relative to the mirror map $R$. Specifically, there exists a constant $\mu>0$ such that, for every $x \in \mathrm{reint}(\mathcal{X})$,
\[
f(x)-f(x^\ast)
\leq
\frac{1}{2\mu}\nu_x^\top \nabla^2 R(x)\nu_x
\leq
\frac{1}{2\mu}\norm{\nu_x}_x^2.
\]
\label{gradient domination}
\end{assumption}
Assumption \ref{gradient domination} is similar to the gradient domination in policy mirror descent algorithm \cite{agarwal2021theory,xiao2022convergence}.
\begin{corollary}
\label{cor:main_convergence}
Fix $\epsilon > 0$. Let $n$ be sufficiently large such that
$n \geq \max\{n_1, n_2\}$, 
where $n_1$ and $n_2$ are the smallest integers satisfying
\[
    f(x_1) + 3 \mathbf{G}^2 
    \leq \frac{\epsilon}{3} 
    \sum_{k=1}^{n_1} \alpha_k,
   \quad \left( 
    L_\nu \mathbf{G}^2 
    + L_{\Pi} \mathbf{G}^2 
    + \frac{L \mathbf{G}^2}{2} 
    \right) 
    \sum_{k=1}^{n_2} \alpha_k^2 
    \leq \frac{\epsilon}{3} 
    \sum_{k=1}^{n_2} \alpha_k,
\]
%
%
respectively. Then, for all $n \geq \max\{n_1, n_2\}$, the following high-probability bound holds:
\begin{equation*}
    \mathbb{P}\Big(\min_{1 \leq k \leq n} \{ f(x_k) - f^\ast \} \geq \epsilon \Big)
    \leq 
    \exp \left( 
    - \frac{\epsilon^2 \left(\sum\limits_{k=1}^{n} \alpha_k\right)^2}
    {18 \mathbf{G}^4 \sum\limits_{k=1}^{n} \alpha_k^2 } 
    \right).
\end{equation*}
\label{concentration - nonconvex PL}
\end{corollary}
\section{Conclusion and Future Work}
We studied stochastic mirror descent under iterate-dependent Markov noise, a setting that naturally arises in stochastic optimization with state-dependent uncertainty. Under mild regularity assumptions, requiring only Lipschitz continuity, we established almost sure convergence for both convex and non-convex objectives. We also derived finite-time concentration guarantees in the smooth setting. In the convex case, the resulting sample complexity matches that of stochastic mirror descent under classical i.i.d.\ noise. For non-convex problems, we obtained finite-time guarantees in terms of a Riemannian gradient measure on the probability simplex, extending existing analyses beyond the standard Euclidean and i.i.d.\ settings.

The results presented in this paper open several avenues for further research. A natural direction is to extend the finite-time analysis to broader classes of non-smooth objectives. Another promising direction is the study of distributionally robust mirror descent through a dynamical systems viewpoint. Finally, it would be of interest to explore the implications of our analysis in application-driven settings, particularly in variants of mirror descent arising in reinforcement learning, such as policy mirror descent.


\bibliographystyle{siamplain}
\bibliography{references}
\end{document}